\documentclass[12pt,reqno]{amsart}
\usepackage{amssymb,url}
\usepackage{units} 
\usepackage[normalem]{ulem} % This package is for strikethroughs as part of editing. We should remove it when we submit
\usepackage{booktabs}
\usepackage[bookmarks=false,unicode,colorlinks,urlcolor=red,citecolor=red,linkcolor=blue]{hyperref}
\usepackage[usenames, dvipsnames]{color}
\usepackage{aliascnt}
\usepackage{enumitem}
\usepackage{tikz}
\usepackage{etoolbox}
\usepackage[margin=1.25in]{geometry}
\usepackage{float}
\usepackage{pgfplots} 
\pgfplotsset{compat=1.8}
\usepackage{doi}
\usepackage{graphicx}
\usepackage{subfigure}
\usepackage{caption}
\makeatletter
\def\tagform@#1{\maketag@@@{\ignorespaces#1\unskip\@@italiccorr}}
\let\orgtheequation\theequation
\def\theequation{(\orgtheequation)}
\makeatother
\def\equationautorefname~{}
\newtheorem{theorem}{Theorem}%[section]

\newaliascnt{lemma}{theorem}

\aliascntresetthe{lemma}

\newaliascnt{proposition}{theorem}
\newtheorem{proposition}[proposition]{Proposition}
\aliascntresetthe{proposition}

\newaliascnt{corollary}{theorem}
\newtheorem{corollary}[corollary]{Corollary}
\aliascntresetthe{corollary}

\newaliascnt{conjecture}{theorem}

\aliascntresetthe{conjecture}

\newaliascnt{example}{theorem}

\aliascntresetthe{example}

\newaliascnt{definition}{theorem}
\newtheorem{definition}[definition]{Definition}
\aliascntresetthe{definition}

\theoremstyle{remark}
\newtheorem*{remark}{Remark}

\newcommand{\RR}{{\mathbb{R}}}

\newcommand{\param}{{(0.4, -0.16, 5, -1, 30)}}

\newcommand{\circled}[2][]{%
  \tikz[baseline=(char.base)]{%
    \node[shape = circle, draw, inner sep = 1pt]
    (char) {\phantom{\ifblank{#1}{#2}{#1}}};%
    \node at (char.center) {\makebox[0pt][c]{#2}};}}
\robustify{\circled}
\newcommand{\RNum}[1]{\uppercase\expandafter{\romannumeral #1\relax}}

\begin{document}
\captionsetup[figure]{labelfont={bf},labelformat={default},labelsep=space,name={Fig.}}
\title[Fourth order Diffusion-Driven Instability]{Diffusion-Driven Instability of a fourth order system}

\author{Jooyeon Chung}
\address{Department of Mathematics, University of Illinois, Urbana, IL 61801, U.S.A.}
\email{jy55266\@@gmail.com}

\date{\today}

\begin{abstract}
We analyze diffusion-driven (Turing) instability of a reaction-diffusion system. The innovation is that we replace the traditional Laplacian diffusion operator with a combination of the fourth order bi-Laplacian operator and the second order Laplacian. We find new phenomena when the fourth order and second order terms are competing, meaning one of them stabilizes the system whereas the other destabilizes it. We characterize Turing space in terms of parameter values in the system, and also find criteria for instability in terms of the domain size and tension parameter.
\end{abstract}

\maketitle

\section{\bf Introduction}
We characterize the Turing space of two-species reaction-diffusion mechanisms with fourth order bi-Laplacian type diffusion. Alan Turing conjectured a mathematical mechanism which explains how two diffusing morphogen populations interact to generate patterns in biology \cite{DMO94, MWBG12, M022, NMMWS03,T52}. This mechanism is now known as Turing instability or diffusion-driven instability. The idea is that two quantities, the activator and inhibitor, satisfy coupled reaction-diffusion equations. These equations admit a linearly stable spatially homogeneous steady state when diffusion is absent, but this homogeneous steady state becomes linearly unstable in the presence of diffusion, initiating a spatially varying inhomogeneous state, or pattern. The space of parameters for which Turing instability occurs is called the Turing space.

In standard Turing analysis, the Laplacian operator $\Delta u = \nabla\cdot\nabla u$
acts for diffusion of the activator and inhibitor, and the domain is fixed. Recent work in Turing's theory has extended applicability of the method, such as by considering growing domains, which are biologically relevant since actual organisms are growing as patterns are forming \cite{KG, MGM10, PSPBM04}. 
In this paper, we will consider a fixed domain but allow the activator and inhibitor to diffuse according to a bi-Laplacian type operator 
\[
\Delta\Delta u -\tau \Delta u
\] that includes both fourth order and second order terms whose relative importance is determined by the tension coefficient. Our analysis applies in all dimensions.

We characterize the parameter values forming the Turing space (Theorem~\ref{thm:turing}), meaning the parameter values for which Turing instability occurs for a given domain. The fourth order situation is different from the standard second order situation because two different types of diffusion, fourth order and second order terms, can compete. The fourth order term stabilizes the system whereas the second order term destabilizes the system, when the tension parameter is negative. Negative tension parameter is considered as destabilization since the backwards Laplacian is ill-posed. This competing situation leads to negative eigenvalues of the diffusion operator, which was not considered in the Laplacian Turing analysis. We show that when competition of two diffusions happens, Turing instability always occurs if we are willing to vary the domain (Corollary~\ref{cor:turing}). One might think it is obvious instability occurs simply because of negative eigenvalues. However, it actually relies upon properties of the spectrum established in the author's paper \cite{CC17}, as we now explain.

To refine our understanding of the Turing space, we identify a cross-section of Turing space in terms of domain and tension parameter (Theorem~\ref{thm:instabilityregion}). We certify for which length of the domain we obtain Turing instability, at least in the one-dimensional case. We fix the reaction parameters but vary the size of domain and tension parameter, and investigate how these changes affect occurrence of Turing instability. This investigation can be done since in \cite{CC17} we analyzed properties of the spectrum of the bi-Laplacian type operator with the natural (free) boundary conditions in one dimension. We find new phenomena when the fourth order and second order terms are competing. Having negative eigenvalues for the diffusion operator does not  by itself make Turing instability occur. Additional conditions need to be satisfied.

To conclude the paper, we apply an analogous cross-sectional Turing analysis to the periodic boundary condition case in one dimension. Even though the periodic boundary condition is not so biologically relevant, it is worth to consider in a sense of providing motivation and insight. The periodic case can be analyzed exactly because the spectrum of the bi-Laplacian type operator for the periodic boundary conditions can be computed exactly. Therefore, we also treat this case and compare the two situations (free and periodic). We find that overall shape of the cross-sectional Turing space for periodic situation is similar to the free case (Figs.~\ref{fig:three_instability_regions} and \ref{fig:periodicinstabilityregions}), which provides insight into the shape of cross-sectional Turing space for the more difficult free case.

All figures presented in this paper were created by the author using the programs \emph{Mathematica} and \emph{Matlab}.

 \subsection*{\bf Related literature}
Although Turing's theory is mostly considered as biological pattern formation, the idea of diffusion-driven instability is not restricted to biology. The mathematical framework can be generally applied wherever the populations can be considered as random moving reactive materials. For instance, researchers have identified Turing-like patterns in the distribution of species in ecological systems, such as the predator-prey model, where the prey acts as activator while the predator acts as inhibitor \cite{AKS11, LHL09, LMP12, NKL95, SEL97}.

\subsubsection*{Growing domain}
It is a natural question to ask how the reaction-diffusion model produces spatial patterns via Turing instability on ``growing" domains. Crampin \emph{et al.\ }were the first researchers to consider the domain growth effects in the reaction-diffusion models \cite{CGM99}. Plaza \emph{et al.\ }\cite{PSPBM04}, Madzvamuse \emph{et al.\ }\cite{MGM10} investigated the role of growth in pattern formation considering Turing instability. For instance, they found that an activator-activator model may give Turing patterns in the presence of domain growth. Such choice of kinetics cannot exhibit Turing instability on fixed domains. Furthermore, a recent paper by Klika and Gaffney \cite{KG} pointed out that analysis of Turing instability on growing domains is even more complicated than Madzvamuse \emph{et al.\ }\cite{MGM10} have considered. They emphasized the history dependence of the stability conditions and the transient nature of the unstable modes with faster growth. An interesting future direction is to apply these conditions for growing domains to the bi-Laplacian type diffusion considered in this paper.

\subsubsection*{Plate problems}
This paper includes analysis using properties of the spectrum of the free rod under tension and compression \cite{CC17}. The rod is the one-dimensional case of the plate. Plate problems are fourth order analogues of membrane problems, with the bi-Laplacian operator taking the place of the Laplacian. The fourth order problems with appropriate boundary conditions have modeled a number of plates with physically relevant conditions. For example, Sweers recently gave a survey of sign- and positivity-preserving properties of rod and plate problems with certain boundary conditions \cite{Sweers(16)}. More recently, Ashbaugh \emph{et al.\ }proved an isoperimetric inequality for the first eigenvalue of the clamped plate under compression for a small range of compression $\tau<0$ \cite{ABM17}. Our investigation in this paper connects the analysis of fourth order plate problem to Turing's model of pattern formation in biology. 

Lewis employed the fourth order type diffusion in a plant-herbivore model \cite{L94}. He showed that the coupling of herbivore dispersal with plant and herbivore dynamics gives rise to both persistent and transient spatial patterns. 

\subsection*{\bf Positivity preservation and thin fluid film diffusion}
Turing instability for fourth order diffusion with a second order term is comprehensively analyzed in this paper. A disadvantage of the fourth order diffusion is that it does not satisfy the minimum principle. Initial data that is positive can evolve to become negative at some point, at a later time, which is not biologically reasonable.

However, the fourth order nonlinear ``thin fluid film equation" that preserves positivity gives a way to solve this problem. For example, 
\begin{equation}\label{eqn:fluid_film}
h_t=-(h^n h_{xxx})_x\pm\tau(h^m h_x)_x
\end{equation} is known to have a ``weak minimum principle" for a sufficiently large value $n$, in that interior finite-time singularities in \eqref{eqn:fluid_film} are forbidden for $n\geq 3.5$ \cite{BBDK94}. Furthermore, Bertozzi and Pugh proved global positivity preservation when $n\geq3.5$ \cite{BP98}. Linearizing such a PDE around a constant steady state gives a linear fourth order PDE of the type considered in my research. Hence the nonlinear ``thin fluid film" PDE with additional reaction terms may be an interesting question for future research in $4$th order pattern formation.

\section{\bf Results on Turing space for fourth order diffusion operator}\label{sec:turingp}
In order to formulate the Turing instability results, we need to set up the reaction-diffusion system, establish notation for the steady state, and specify the boundary conditions and eigenvalues of the diffusion operator.

The interaction of two chemicals, activator $u$ and inhibitor $v$, gives a reaction-diffusion system of equations 
\begin{align} 
\frac{\partial u}{\partial t} &= (-\Delta^2 u + \tau\Delta u) +  f(u, v) \label{eqn:activator}\\
\frac{\partial v}{\partial t} &= k(-\Delta^2 v + \tau\Delta v) +  g(u, v), \label{eqn:inhibitor}
\end{align}
where the Laplacian is
\[
\Delta u=\frac{\partial^2 u}{\partial x_1^2} + \dots + \frac{\partial^2 u}{\partial x_n^2},
\] the bi-Laplacian is
\[
\Delta^2 u=\Delta\Delta u,
\]
$\tau$ is a ``tension" coefficient, $k>0$ is a proportionality constant of diffusion (the ``diffusivity"), and $f$ and $g$ model the reaction kinetics. The bi-Laplacian type operator $\Delta^2-\tau\Delta$ includes both $4$th order and $2$nd order terms whose relative importance is determined by $\tau$. Even though the terminology is related to the vibrating plate model \cite[Section 2]{C11} and it is not relevant to diffusion, we still call $\tau$ the tension coefficient. 

We fix the homogeneous steady state $(u_0, v_0)\in \RR^2$ of \eqref{eqn:activator}--\eqref{eqn:inhibitor} to be the solution of
\begin{equation*}
f(u_0, v_0)=0, \quad g(u_0, v_0)=0,
\end{equation*} and the partial derivatives of $f$ and $g$ to be evaluated at the steady state $(u_0, v_0)$, so that
\[
f_u=f_u(u_0, v_0),\quad f_v=f_v(u_0, v_0),
\] throughout the paper, and similarly for $g_u$ and $g_v$.

For the domain $\Omega \subset \mathbb{R}^n$, we work with the natural (free) boundary conditions associated with the diffusion operator $\Delta^2-\tau\Delta$. In dimension $n=2$, this means $u$ and $v$ satisfy boundary conditions of the type
\begin{align}
&\frac{\partial^2\phi}{\partial n^2}=0 &\text{on  }& \partial\Omega,\label{BC1}\\
&\tau\frac{\partial \phi}{\partial n} - \frac{\partial(\Delta \phi)}{\partial n} - \frac{\partial}{\partial s}\left(\frac{\partial^2\phi}{\partial s\partial n} - K(s)\frac{\partial \phi}{\partial s}\right) = 0 &\text{on  }& \partial\Omega,\label{BC2}
\end{align} where $n$ denotes outward unit normal derivative, $s$ the arclength, and $K$ the curvature of $\partial\Omega$. For $n$-dimension, the natural (free) boundary conditions for $\Delta^2-\tau\Delta$ are stated in \cite[Proposition 5]{C11}. The natural boundary condition \eqref{BC2} with $\phi=u$ and $\phi=v$ imply that mass is conserved by the diffusion operator. 

The eigenvalues $\mu_j=\mu_j(\Omega, \tau)$ of the operator $\Delta^2-\tau\Delta$ are governed by the differential equation
\begin{equation}\label{DE}
\Delta^2u-\tau \Delta u=\mu u
\end{equation} together with the natural boundary conditions \eqref{BC1}--\eqref{BC2} on $\Omega$, and are listed in increasing order as
\[
\mu_1\leq \mu_2\leq\mu_3\leq\cdots \to \infty.
\] There is always a zero eigenvalue, with constant eigenfunction. When $\tau\geq 0$, this zero eigenvalue is the lowest eigenvalue. When $\tau<0$, there is at least one negative eigenvalue. For more on the spectrum and the relevant Sobolev spaces and bilinear forms, see \cite{C11, CC17}.

\begin{remark} Imposing Dirichlet boundary conditions would cause a flux of $u$ and $v$ through the boundary, so there might be some loss of spatial patterns. Therefore we do not consider Dirichlet conditions. On the other hand, we will investigate a simpler boundary condition at the end of the paper, that is, the periodic boundary conditions in one dimension.
\end{remark}

\noindent Notice we have the same diffusion operator for both activator and inhibitor, up to constant multiple. Hence we can expand both $u$ and $v$ in terms of the same eigenfunctions, to carry out the Turing instability analysis, in Section~\ref{sec:pfthm_turing}.\\

In the next definition, we need a system of ordinary differential equations
\begin{align} 
\frac{\partial u}{\partial t} &= f(u, v) \label{ODE1}\\
\frac{\partial v}{\partial t} &= g(u, v), \label{ODE2}
\end{align} which is same as the system \eqref{eqn:activator}--\eqref{eqn:inhibitor} without the diffusion terms.

\begin{definition}[Turing space]
Consider a smoothly bounded domain $\Omega\subset \RR^n$. The reaction diffusion system \eqref{eqn:activator}--\eqref{BC2} admits \emph{Turing instability} if the homogeneous steady state $(u_0, v_0)$ is linearly asymptotically stable to small perturbations in the absence of diffusion (meaning for the ODE system \eqref{ODE1}--\eqref{ODE2}), but linearly unstable to small \emph{spatial} perturbations when diffusion is present (meaning for the PDE system \eqref{eqn:activator}--\eqref{BC2}).

The \emph{Turing space} for $\Omega$ is the space of parameters giving \emph{Turing instability}:
\begin{align*}%\label{TS}
\begin{split}
TS(\Omega)=\{(f_u, f_v, g_u, g_v,& k, \tau)\in\RR^6: \text{the homogeneous steady state $(u_0, v_0)$}\\
 &\text{is linearly asymptotically stable in the absence of diffusion}\\
&\qquad\text{ but unstable when diffusion is present} \}.
\end{split}
\end{align*} For convenience, we use notation $\vec{p}=(f_u, f_v, g_u, g_v, k)\in\RR^5$ as vector of reaction-diffusion parameters. With fixed $\tau$, the \emph{Turing space} for $\Omega$ and $\tau$ is the cross-section
\begin{align*}
TS(\Omega, \tau)&=\{\vec{p}\in\RR^5: (\vec{p}, \tau)\in TS(\Omega)\}.
\end{align*}
\end{definition}

In this section, we will fix $\tau$ and find conditions for the reaction-diffusion parameter vector $\vec{p}$ to get a Turing instability on the domain $\Omega$.

%Since we are interested in linearly instability of the steady state that is only \emph{spatially} dependent, in the absence of any spatial variation the homogeneous steady state $(u_0, v_0)$ must be linearly stable. We first derive conditions for these to hold and then determine conditions for the spatially inhomogeneous instability. 

We define three quantities which will be used in the following discussion: 
\begin{align}
A(\vec{p})&=\frac{f_u+g_v}{1+k},\label{A}\\
\begin{split}
a(\vec{p}), b(\vec{p})&=\frac{(kf_u+g_v)\pm\sqrt{(kf_u+g_v)^2-4k(f_ug_v-f_vg_u)}}{2k},\label{mupm}\\
\text{where $a$ cor}&\text{responds to the minus root and $b$ corresponds to the plus root.}
\end{split}
\end{align}
Recall that $\mu_j=\mu_j(\Omega, \tau)$ denotes the $j$th eigenvalue of the diffusion operator $\Delta^2-\tau\Delta$ with natural boundary conditions \eqref{BC1}--\eqref{BC2} on the domain $\Omega$. Define
\[
\operatorname{Spec}(\Omega, \tau)=\text{spectrum}=\{\mu_j(\Omega, \tau): j=1, 2, 3, \dots\}.
\]
\begin{theorem}[Characterization of Turing space for fixed domain]\label{thm:turing}
Given the domain $\Omega$, if $\tau\geq 0$ then the Turing space is
\[
TS(\Omega, \tau) =
\{ \vec{p} \in \mathbb{R}^5 : \text{ \eqref{firstcond}--\eqref{fourthcond} hold, and $\operatorname{Spec}(\Omega, \tau) \cap\left(a(\vec{p}), b(\vec{p})\right)\neq \emptyset$}\},
\] and if $\tau<0$ then the Turing space is
\begin{align*}
\begin{split}
TS(\Omega, \tau) =\{\vec{p} \in \mathbb{R}^5 : &\text{ either \eqref{firstcond}--\eqref{thirdcond} hold and $\operatorname{Spec}(\Omega, \tau) \cap\left(a(\vec{p}), b(\vec{p})\right)\neq \emptyset$,}\\
&   \text{    or \eqref{firstcond}--\eqref{secondcond} hold and $\mu_1< A(\vec{p})$}\},
\end{split}
\end{align*}
where the conditions are
\begin{align}
f_u + g_v < 0, \label{firstcond}\\
f_u g_v - f_v g_u > 0, \label{secondcond}\\
(kf_u + g_v)^2 - 4 k(f_u g_v - f_v g_u) > 0, \label{thirdcond}\\
kf_u + g_v > 0.\label{fourthcond}
\end{align} 
\end{theorem}
Note condition \eqref{firstcond} implies $A(\vec{p})<0$.

\begin{remark}
When $\tau\geq0$, conditions \eqref{firstcond} and \eqref{fourthcond} imply $k\neq 1$, meaning one of the activator and inhibitor must diffuse faster than the other. When $\tau<0$, since the condition \eqref{fourthcond} need not be assumed in Theorem~\ref{thm:turing}, we see $k$ can equal $1$, meaning the activator and inhibitor possibly diffuse at the same rate.
\end{remark}

We characterized the Turing space for a fixed domain $\Omega$ in Theorem~\ref{thm:turing}. When $\tau\geq0$, all four conditions \eqref{firstcond}--\eqref{fourthcond} are the same as the standard Turing space for the Laplacian. When $\tau<0$, only the first two or three of these conditions are required to be in the Turing space, so the Turing space for the fourth order operator $\Delta^2 - \tau\Delta$ is larger than standard Turing space. 

In the following corollary, we show that we can always get Turing instability when $\tau<0$, provided we are willing to vary the domain. The theorem is proved in Section~\ref{sec:pfthm_turing}, and its corollary in Section~\ref{sec:pfthm_muast,pfcor_turing}.

\begin{corollary}\label{cor:turing}
If $\tau<0$ and $\vec{p}$ satisfies \eqref{firstcond} and \eqref{secondcond}, then Turing instability occurs for some domain $\Omega$.
\end{corollary}

Corollary~\ref{cor:turing} depends on a certain fact about the minimum of the lowest eigenvalue of $\Delta^2-\tau\Delta$, which we state below as Theorem~\ref{thm:muast}.

Write $D^2u$ for the Hessian matrix of $u$, and $|D^2u|^2=\sum_{i, j} u^2_{x_i, x_j}$.
\begin{definition}\label{def:muast}
Define $\mu_1 =$ lowest eigenvalue of $(\Delta^2 - \tau\Delta) u = \mu u$ with natural (free) boundary conditions \eqref{BC1}--\eqref{BC2}. That is (see \cite[Section 2]{C11}),
\begin{align}\label{mu1}
\mu_1(\Omega, \tau) = \min_{u\in H^2(\Omega)}\frac{\int_{\Omega}\big(|D^2u|^2 + \tau |\nabla u|^2\big)\, dx}{\int_{\Omega} u^2\, dx}.
\end{align} Denote the ``smallest possible" first eigenvalue by
\begin{align*}
\mu_1^{\ast}(\tau)  = \inf_{\Omega}\mu_1(\Omega, \tau).
\end{align*} The ratio on the right of \eqref{mu1} is called the Rayleigh quotient. It is obtained formally by multiplying the eigenvalue equation \eqref{DE} by $u$ and integrating by parts, using the natural boundary conditions \eqref{BC1}--\eqref{BC2}.
\end{definition}
Observe that
$\mu_1^{\ast}(\tau) < 0$ when $\tau < 0$, by choosing a linear trial function. We show $\mu_1^{\ast}(\tau)=-\infty$, which is the key to proving Corollary~\ref{cor:turing}.

\begin{theorem} \label{thm:muast}
If $\tau<0$ then the first eigenvalue can be arbitrarily negative:
\begin{equation*}
\mu_1^{\ast}(\tau) = -\infty.
\end{equation*}
\end{theorem}

%Theorem~\ref{thm:muast} allows the Turing space in Theorem~\ref{thm:turing} for $\tau<0$ to be simplified.

\section{\bf Instability regions of the fourth order diffusion operator $\Delta^2-\tau\Delta$ in one dimension}\label{sec:turingRtau}

In this section, we consider a different cross-section of Turing space: we will look at which combinations of the size of domain and the tension parameter $\tau$ produce Turing instability when the reaction-diffusion parameters are fixed. 

In this section we restrict attention to one dimension, since earlier work \cite{CC17} gives detailed information on the spectrum of the diffusion operator $\Delta^2-\tau\Delta$ in one dimension. The domain is the interval
\[
\Omega(R)=(-R, R).
\] We introduce the Turing spaces with fixed $\vec{p}$:
\begin{definition}[Turing space with fixed parameter]\label{def:crosssectionalTuring}
\begin{align*}
TS(\vec{p})=\{(R, \tau)\in\RR^2: \vec{p}\in TS\left(\Omega(R), \tau\right), R>0\}.
\end{align*}
\end{definition}
%Given $\Omega$ and $\tau$, we obtained the necessary and sufficient conditions for the formation of spatial patterns by two quantities reaction diffusion mechanisms of the form \eqref{eqn:activator}--\eqref{BC2} in Section~\ref{sec:turingp}. More than those conditions \eqref{firstcond}--\eqref{fourthcond} for $\tau\geq 0$ and either \eqref{firstcond}--\eqref{fourthcond} or \eqref{firstcond}--\eqref{secondcond}  for $\tau<0$, we see that some values $\mu_j$ should correspond to eigenvalues of the diffusion operator $\Delta^2 - \tau\Delta$ for the generation of Turing instability. The further requirement induces relation between the size of domain $R$ and the tension parameter $\tau$. 
This definition produces a region in $(R, \tau)$-plane and our goal is to determine the shape of this region (see Fig.~\ref{fig:three_instability_regions}) and to understand some of its properties. We have seen in \cite{CC17} that the spectrum of the operator $\Delta^2-\tau\Delta$ can be split into eigenvalue branches $\mu_{l}^{\text{odd}}(\tau)$ and $\mu_{l}^{\text{even}}(\tau)$ depending on $\tau$ and an index $l\geq 0$ and also depending on the evenness and oddness of the underlying eigenfunction. See details in \cite{CC17} and Fig.~\ref{fig:freeeigenvalues}. For each corresponding eigenvalue branch we will define two regions in $(R, \tau)$-plane and then we will prove in Theorem~\ref{thm:instabilityregion} that those regions are the instability regions. That is, pairs $(R, \tau)$ in these regions are the length and tension parameters which give points in the Turing space, for the fixed reaction-diffusion parameter vector $\vec{p}$.

\begin{definition}[Instability region]\label{def:instabilityregion}
Fix a reaction-diffusion parameter vector $\vec{p}$ that satisfies condition \eqref{firstcond}, and recall the number $A=A(\vec{p})$ from \eqref{A}, noting $A<0$ by \eqref{firstcond}. Define regions
\begin{align*}
E_{-}(l) &= \{\left(R, \tau\right): R^{-4}\mu_{l}^{even}(\tau R^2)< A\ \text{and}\   \tau<0\},\\
O_{-}(l) &= \{\left(R, \tau\right): R^{-4}\mu_{l}^{odd}(\tau R^2)< A\ \text{and}\   \tau<0\},
\end{align*} for $l\geq 0$. If in addition $\vec{p}$ satisfies condition \eqref{thirdcond} then the numbers $a=a(\vec{p})$ and $b=b(\vec{p})$ in \eqref{mupm} make sense, and we define
\begin{align*}
E_{+}(l) &= \{\left(R, \tau\right):  R^{-4}\mu_{l}^{\text{even}}(\tau R^2)\in(a, b)\},\\
O_{+}(l) &= \{\left(R, \tau\right):  R^{-4}\mu_{l}^{\text{odd}}(\tau R^2)\in(a, b)\}.
\end{align*} (The $``+"$ and $``-"$ notation refers to the sign of the unstable eigenvalues in the proof of Theorem~\ref{thm:instabilityregion} below.) Let $E=E_{+}\cup E_{-}$ and $O=O_{+}\cup O_{-}$, for each $l$. 
\end{definition}

Fig.~\ref{instability_region} shows the instability regions $E(0)$ and $O(0)$ associated to the zero-th even and odd eigenvalue branches, respectively. These figures were formed using implicit parameterizations of the eigenvalue branches $\mu_0^{\text{even}}$ and $\mu_0^{\text{odd}}$, respectively, as described at the end of the section.

\begin{figure}
\includegraphics[scale=0.5]{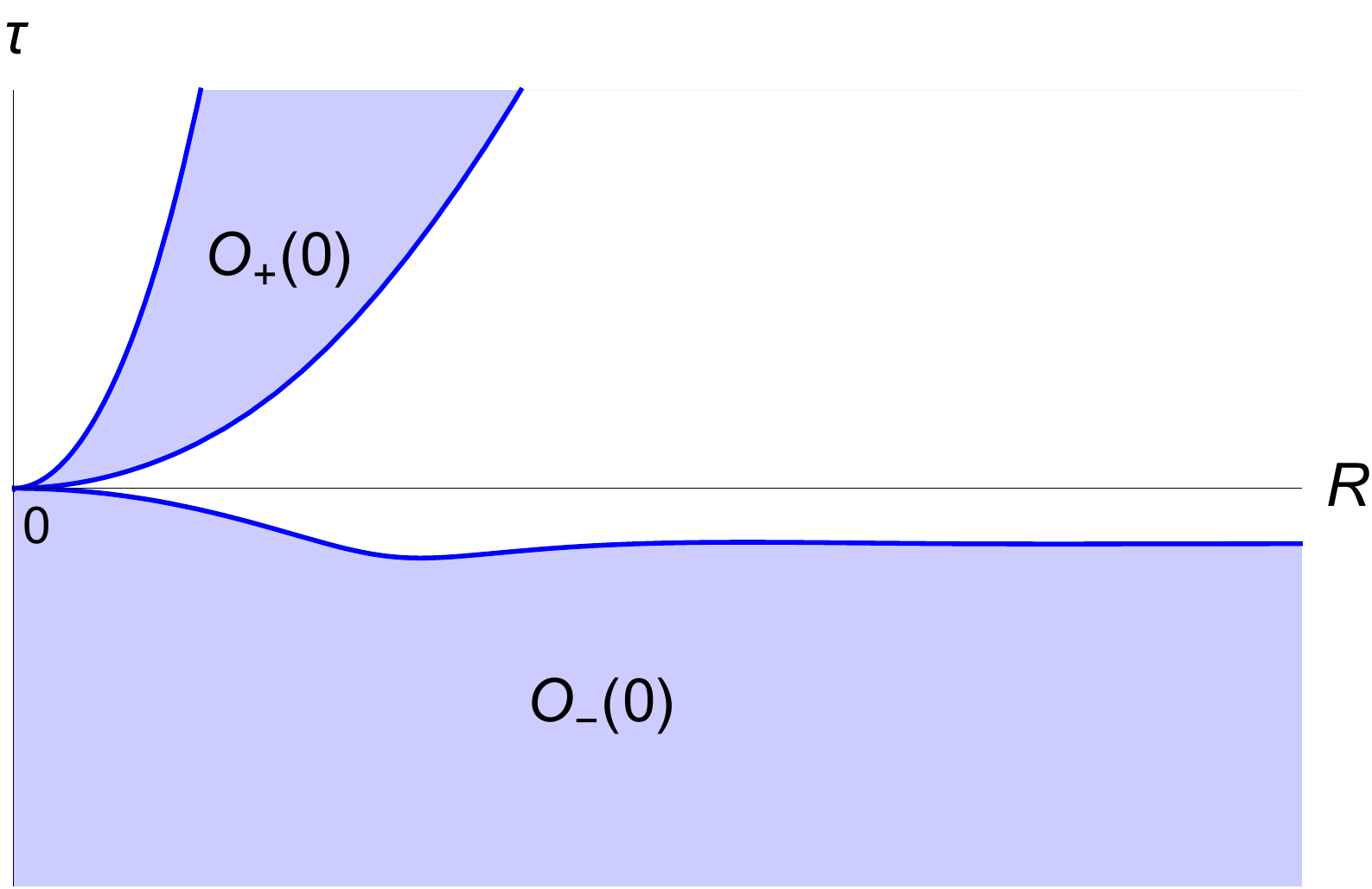}
\quad
\includegraphics[scale=0.5]{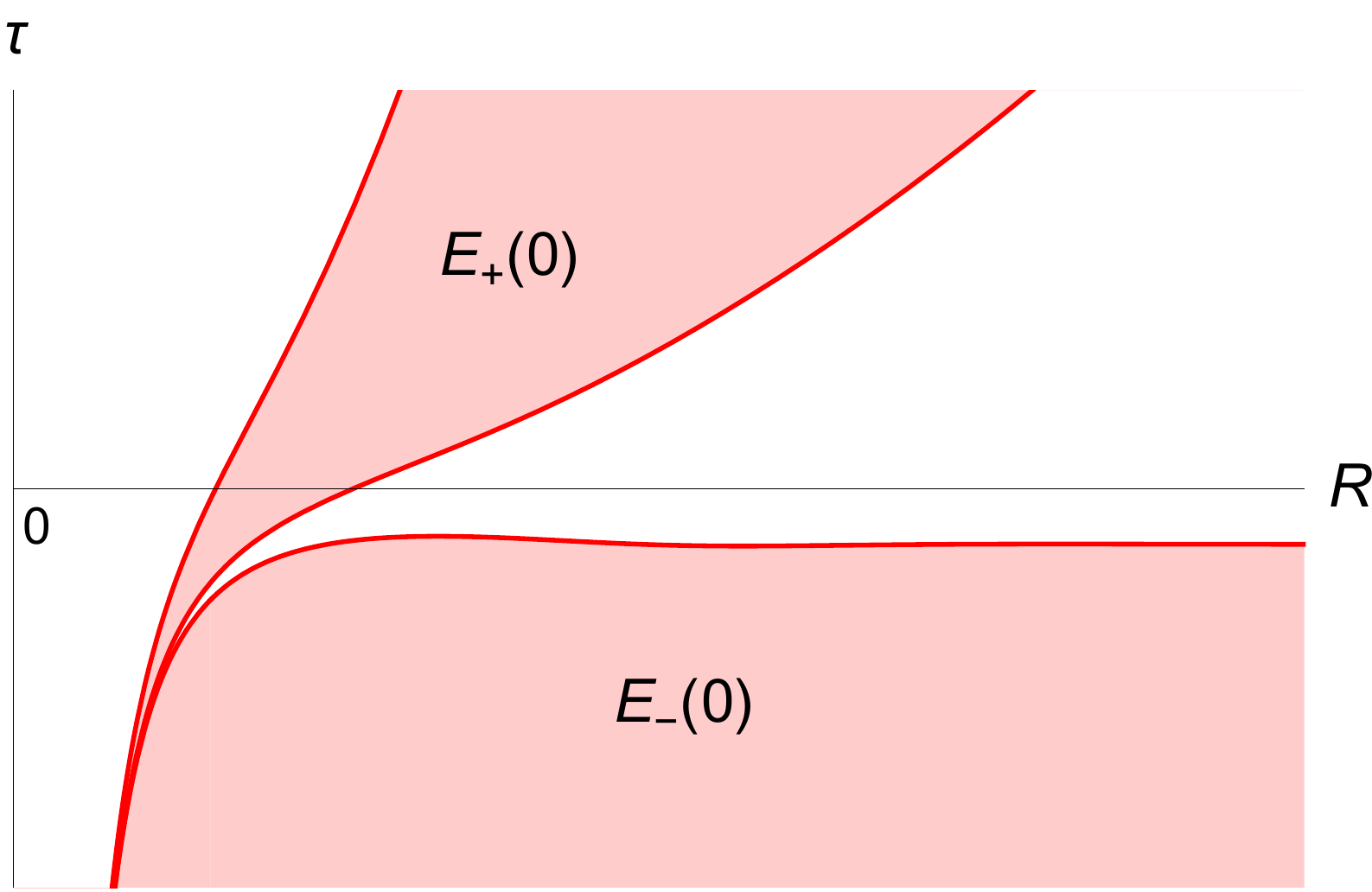}
\caption{\label{instability_region}Points ($R, \tau$) in shaded regions mean that for the interval $\Omega(R)=(-R, R)$, the stable homogeneous steady state of the reaction system becomes unstable in the presence of diffusion. Here we assume the reaction-diffusion vector $\vec{p}$ satisfies conditions \eqref{firstcond}--\eqref{fourthcond}. (The figure uses the Gierer--Meinhardt system \cite[Section 2.2]{M022} and parameter values $\vec{p}=\param$.) The first (resp.\ second) figure describes the instability region associated to the zero-th odd (resp.\ even) eigenvalue branch of $\Delta^2u - \tau\Delta u=\mu u$, as explained immediately after Definition~\ref{def:instabilityregion}. See also Fig.~\ref{fig:three_instability_regions}.}
\end{figure}

In the next theorem, we will show that the instability regions we have found make up the whole Turing space $TS(\vec{p})$. %Let the quadrants be
%\begin{align*}
%Q_1=\{(R, \tau): R>0, \tau\geq0\},\quad Q_4=\{(R, \tau): R>0, \tau<0\}.
%\end{align*} 
Remember that $E_{+}$ intersects both first and fourth quadrants in the $(R, \tau)$-plane, while $E_{-}$ lies in the lower (fourth) quadrant, and similarly for $O_{+}$ and $O_{-}$. Recall $E=E_{+}\cup E_{-}$ and $O=O_{+}\cup O_{-}$.

\begin{theorem}[Instability region associated to each eigenvalue branch]\label{thm:instabilityregion}\

\begin{enumerate}
\item \label{part1}If the reaction-diffusion vector $\vec{p}$ satisfies \eqref{firstcond}--\eqref{fourthcond}, then the regions $E$ and $O$ generate Turing instability and fill the Turing space $TS(\vec{p})$:
\begin{equation*}
TS(\vec{p}) = \bigcup_{l\geq 0} \big(E(l) \cup O(l)\big).
\end{equation*}
\item If the reaction-diffusion vector $\vec{p}$ satisfies \eqref{firstcond}--\eqref{secondcond}, then the regions $E_{-}$ and $O_{-}$ generate Turing instability:
\begin{equation*}
E_{-}(l)\cup O_{-}(l) \subset TS(\vec{p}), \quad l\geq 0.
\end{equation*}
\end{enumerate}
\end{theorem}
In other words, if a pair $(R, \tau)$ belongs to the instability region $E(l)$ or $O(l)$, $l\geq0$, then Turing instability occurs for the domain $\Omega(R)=(-R, R)$ with the tension parameter $\tau$. That is, the spatially homogeneous linearly asymptotically stable  steady state $(u_0, v_0)$ of \eqref{ODE1}--\eqref{ODE2} becomes unstable under diffusion.

%\begin{corollary}[Instability regions fill up the Turing space]{\color{red} I think this corollary statement still need to be fixed because of confusion of assumption!}\

%\begin{enumerate}
%\item If the reaction-diffusion vector $\vec{p}$ satisfies \eqref{firstcond}--\eqref{fourthcond}, then the instability regions $E_{+}(l)$ and $O_{+}(l)$, $l\geq0$ fill the Turing space $TS(\vec{p})$ in the first and fourth quadrants: 
%\begin{align*}
%Q_1 \cap TS(\vec{p})= Q_1 \cap \bigcup_{l\geq 0} \{E_{+}(l) \cup O_{+}(l)\},\\
%Q_4\cap TS(\vec{p}) = Q_4 \cap \bigcup_{l\geq 0} \{E_{+}(l) \cup O_{+}(l)\}.
%\end{align*}
%\item If the reaction-diffusion vector $\vec{p}$ satisfies \eqref{firstcond}--\eqref{secondcond}, then the instability regions $E_{-}(l)$ and $O_{-}(l)$, $l\geq0$ fill the Turing space $TS(\vec{p})$ in the fourth quadrant: 
%\begin{align*}
%Q_4\cap TS(\vec{p}) = Q_4 \cap \bigcup_{l\geq 0} \{E_{-}(l) \cup O_{-}(l)\}.
%\end{align*}
%\end{enumerate}
%\end{corollary}

We found infinitely many instability regions $E_{\pm}(l)$, $O_{\pm}(l)$ in Theorem~\ref{thm:instabilityregion}. We will discuss how these regions behave as $l$ increases in the following Proposition~\ref{prop:moverightdown}.

\begin{proposition}[Movement of instability regions as index $l$ increases]\label{prop:moverightdown}\
\begin{enumerate}
\item \label{prop:moverightdown_1}Assume \eqref{thirdcond} holds. The regions  $E_{+}(l)$ and $O_{+}(l)$ move downwards as the index $l$ increases, in the sense that the top (resp.\ bottom) boundary curve of region $E_{+}(l)$ lies above the top (resp.\ bottom) boundary curve of region $E_{+}(l+1)$.
\item Assume \eqref{firstcond} holds. The regions $E_{-}(l)$ and $O_{-}(l)$ are nested as $l$ increases:
\begin{align*}
E_{-}(0)\supset E_{-}(1) \supset E_{-}(2)\supset \cdots,\\
O_{-}(0)\supset O_{-}(1) \supset O_{-}(2)\supset \cdots.
\end{align*}
\end{enumerate}
\end{proposition}
Fig.~\ref{fig:three_instability_regions} shows some of these regions, and the nesting behavior.

\begin{remark}
One might think it is obvious that we get Turing instability if $\tau<0$ because of existence of negative eigenvalues for $\Delta^2-\tau\Delta$. However, we show in the next corollary that having negative eigenvalue is not always enough to get instability. We have some region in the $(R, \tau)$-plane with $\tau<0$ which corresponds to homogeneous steady states of the reaction-diffusion system staying stable. This stability relies upon a certain fact about the growth rate of the spectrum of $\Delta^2 - \tau\Delta$ with free boundary conditions when $\tau$ is small negative. 
\end{remark}

\begin{figure}
\includegraphics[scale=0.35]{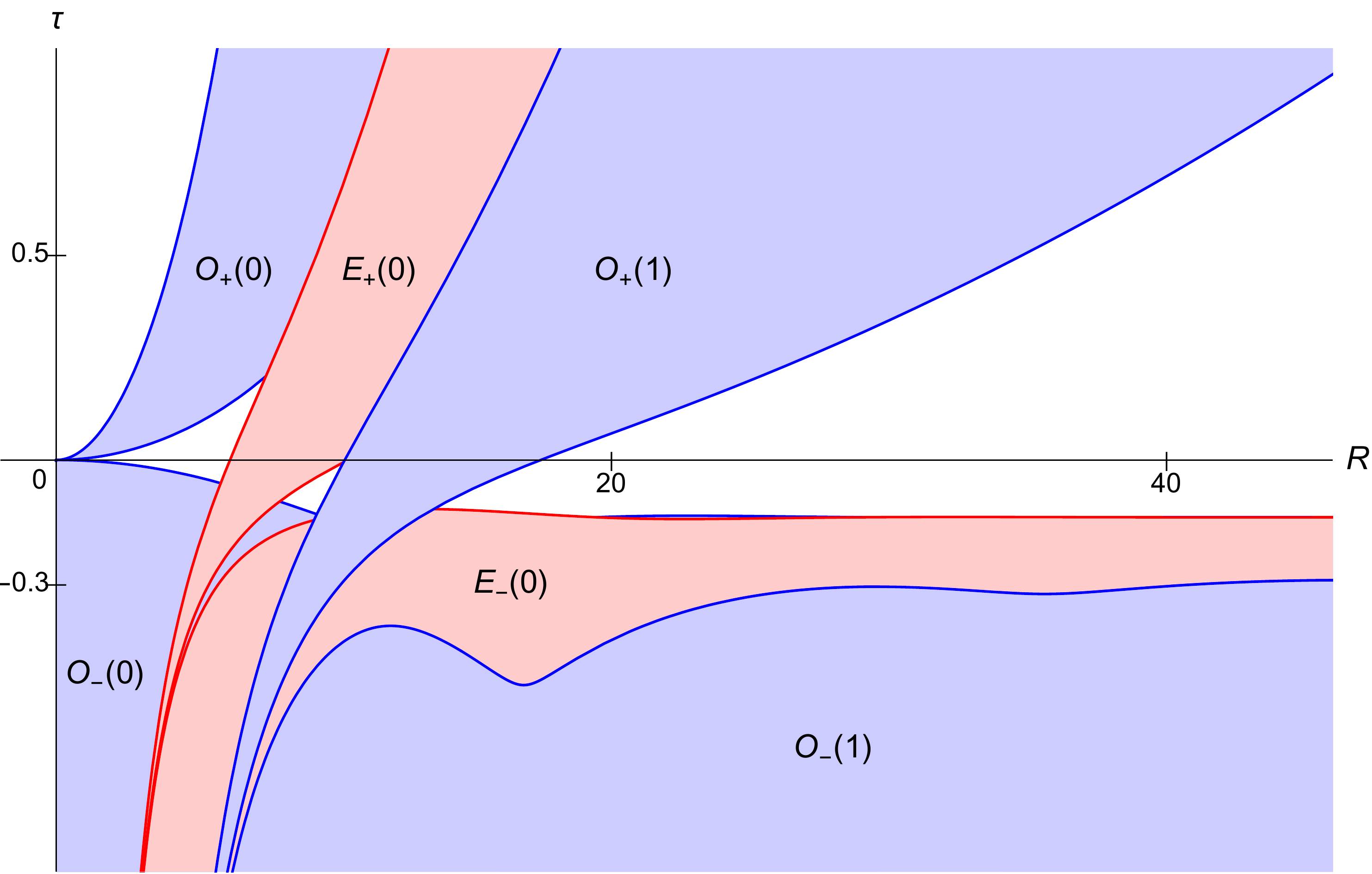}
\caption{\label{fig:three_instability_regions}The instability regions $O_{\pm}(0)$, $E_{\pm}(0)$, and $O_{\pm}(1)$, assuming conditions \eqref{firstcond}--\eqref{fourthcond} hold for the reaction-diffusion vector $\vec{p}$. (The figure uses parameter values $\vec{p}=\param$.) Observe some parts of the lower half plane are not covered by any instability regions.}
\end{figure}

\begin{corollary}[Existence of region outside the Turing space]\label{prop:stableregion}
If the reaction-diffusion vector $\vec{p}$ satisfies \eqref{firstcond}--\eqref{fourthcond}, then there exists some region in $(R, \tau)$-plane where $\tau< 0$ that is outside the union $\displaystyle{\bigcup_{l\geq 0} \big(E(l) \cup O(l)\big) }$ of the instability regions.
\end{corollary}

Fig.~\ref{fig:three_instability_regions} shows these regions. The corollary says there is some unshaded part in the lower half plane.

\subsection*{Extra instability regions when $\tau<0$}\label{sec:negative_instability_region}
In this subsection, we describe some additional instability regions when $\tau<0$.
From Theorem~\ref{thm:turing}, there are two cases for the reaction-diffusion vector $\vec{p}$ belonging to the Turing space when $\tau<0$:
\begin{align}
\text{either   } &\text{\eqref{firstcond}--\eqref{thirdcond} hold and $\operatorname{Spec}(\Omega, \tau) \cap\left(a(\vec{p}), b(\vec{p})\right)\neq \emptyset$,}\label{firstcase}\\
\text{or        } &\text{\eqref{firstcond}--\eqref{secondcond} hold and $\mu_1< A(\vec{p})$}\label{secondcase}.
\end{align} The instability regions $E_{-}$ and $O_{-}$ arise from the case \eqref{secondcase}, as shown in Definition~\ref{def:instabilityregion} and Theorem~\ref{thm:instabilityregion}. In addition to these regions shown in Fig.~\ref{fig:three_instability_regions}, in this subsection we will describe what the instability regions arising from the case \eqref{firstcase} look like. 

In the traditional Turing analysis with the Laplacian, the Turing space would be empty if \eqref{fourthcond} fails (that is, if $kf_u+g_v<0$), because $a(\vec{p})$ and $b(\vec{p})$ are negative while the spectrum of the Laplacian is positive. But $\Delta^2-\tau\Delta$ permits negative eigenvalues when $\tau<0$. This introduces extra instability regions (as shown in Fig.~\ref{fig:instability_regions_negative15}), i.e., creates some Turing space.

Assume the reaction-diffusion vector $\vec{p}$ satisfies \eqref{firstcond}--\eqref{thirdcond} and $kf_u+g_v<0$ (meaning \eqref{fourthcond} fails). Define 
\begin{align*}
\widetilde{E}(l) &= \{\left(R, \tau\right):  R^{-4}\mu_{l}^{\text{even}}(\tau R^2)\in(a, b)\ \text{and}\ \tau<0\},\\
\widetilde{O}(l) &= \{\left(R, \tau\right):  R^{-4}\mu_{l}^{\text{odd}}(\tau R^2)\in(a, b)\ \text{and}\ \tau<0\}.
\end{align*} Note that \eqref{thirdcond} guarantees the numbers $a=a(\vec{p})$ and $b=b(\vec{p})$ in \eqref{mupm} make sense, and these numbers are negative because $kf_u+g_v<0$. We present the regions $\widetilde{E}(l)$ and $\widetilde{O}(l)$ numerically in Fig.~\ref{fig:instability_regions_negative15}. The regions are obtained in a similar way to regions $E_{\pm}$ and $O_{\pm}$. We use an implicit parameterization for $\mu_l^{\text{even}}$ and $\mu_l^{\text{odd}}$ in terms of two other parameters \cite[Theorem 11 and Lemma 13]{CC17}, but now we only need to consider eigenvalue branches in the lower half of the spectral plane.

Unlike the instability regions $E_{-}(l)$ and $O_{-}(l)$ in Definition~\ref{def:instabilityregion} that only have upper boundary curves, the regions $\widetilde{E}(l)$ and $\widetilde{O}(l)$ have upper and lower boundary curves. To sum up, negative $\tau$ values introduce negative eigenvalues for the diffusion operator $\Delta^2-\tau\Delta$ which lead to the appearance of some instability regions. These are relatively smaller than the regions $E_{-}(l)$ and $O_{-}(l)$ in Definition~\ref{def:instabilityregion}.

%However, the Turing conditions are still difficult to satisfy and so most values of $R$ and $\tau$ are not in the Turing space. To sum up, the negative $\tau$ introduces negative eigenvalues for the diffusion operator $\Delta^2-\tau\Delta$ which lead to the appearance of some instability regions but these are relatively small.

\begin{figure}
\includegraphics[scale=0.5]{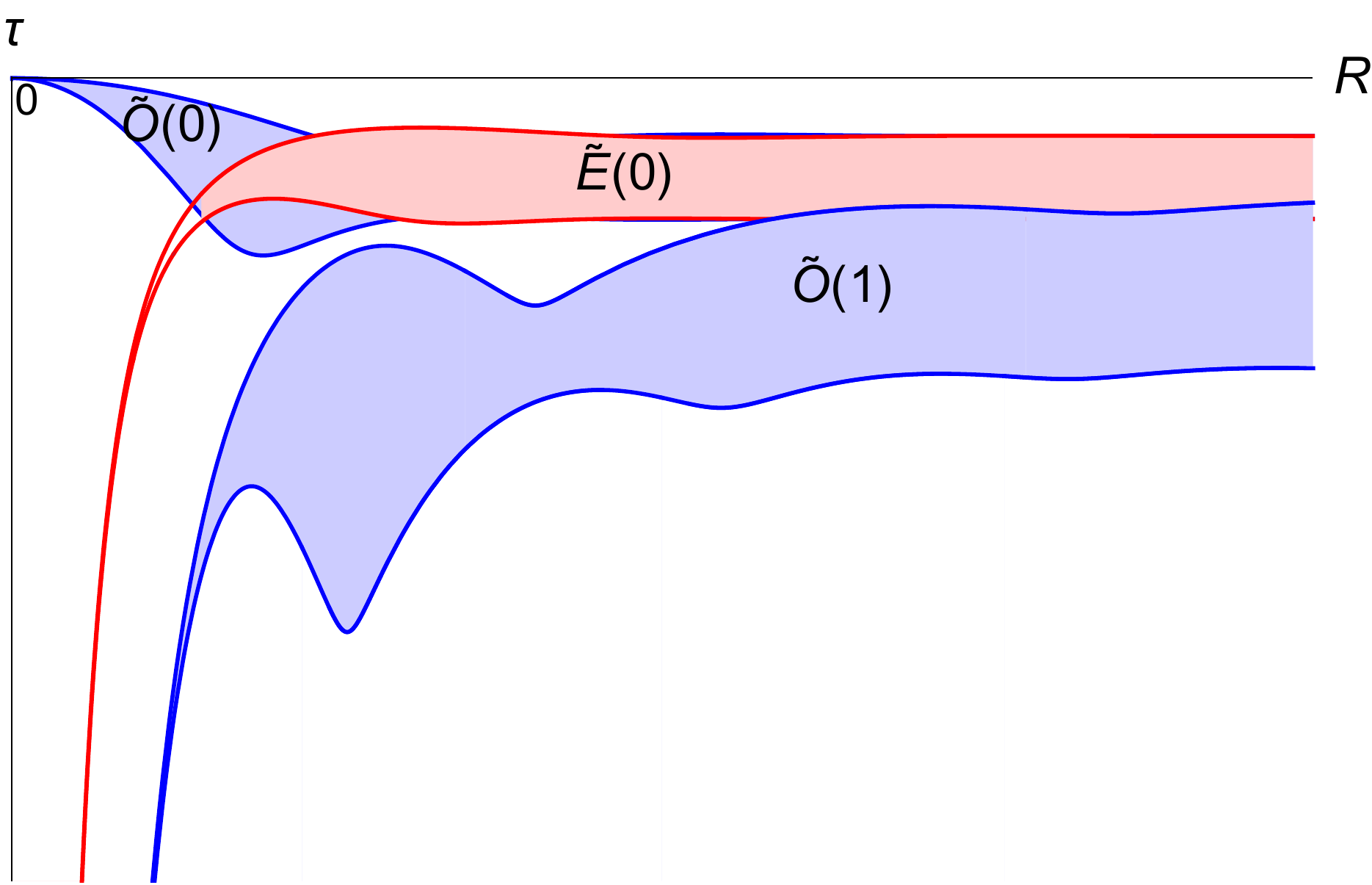}
\caption{\label{fig:instability_regions_negative15}New phenomenon: the negative eigenvalues of $\Delta^2u-\tau\Delta u=\mu u$ permit Turing instability even when $kf_u+g_v<0$ (meaning \eqref{fourthcond} fails), whereas in the traditional Turing analysis with the Laplacian governing diffusion, the Turing space would be empty. In the pictured situation, conditions \eqref{firstcond}--\eqref{thirdcond} hold for the reaction-diffusion vector $\vec{p}$ and condition \eqref{fourthcond} fails. (The figure uses parameter values $\vec{p}=(0.1, -0.01, 20, -1, 1)$.) The regions of $\widetilde{O}(0)$, $\widetilde{E}(0)$, and $\widetilde{O}(1)$ are associated to the eigenvalue branches $\mu_0^{\text{odd}}$, $\mu_0^{\text{even}}$, and $\mu_1^{\text{odd}}$, respectively, in the lower half of the spectral plane (see earlier explanation in the subsection).}
\end{figure}

\subsection*{Numerical experiments.} We do some numerical simulations, to see beyond the linear predictions from spectral theory to what is happening in the genuinely nonlinear regime. We modify Gierer and Meinhardt's reaction kinetics \cite[Equation $(15)$]{GM72}, \cite[Section $2.2$]{M022}, to use the fourth order diffusion $\Delta^2-\tau \Delta$ on the interval $\Omega(R)=(-R, R)$ in $1$-dimension. The Gierer--Meinhardt reaction system is
\begin{equation}\label{G-Mreaction}
f(u, v)=k_1-k_2 u+\frac{k_3u^2}{v}, \quad g(u, v)=k_4u^2-k_5v.
\end{equation} We fix constants $(k_1, k_2, k_3, k_4, k_5)=(0, 0.4, 1, 1, 1)$ for our numerical simulations. Hence the homogeneous steady state $(u_0, v_0)$ is
\[
(u_0, v_0)=(2.5, 6.25),%=\left(k_2^{-1}, k_2^{-2}\right), 
\] and the partial derivatives of $f$ and $g$ evaluated at the steady state $(u_0, v_0)$ are
\[
(f_u, f_v, g_u, g_v)=(0.4, -0.16, 5, -1).%=\left(k_2, -k_2^2, 2k_2^{-1}, -1\right)
\] We also take the diffusivity $k=30$. Figs.~\ref{fig:patterns_positive} and \ref{fig:patterns_negative} illustrate inhomogeneous steady states corresponding to points in the instability regions $O_{+}(1)$ and $E_{-}(0)$ in Theorem~\ref{thm:instabilityregion} part (\ref{part1}). An unstable steady state for $\tau\geq0$ is illustrated in Fig.~\ref{fig:patterns_positive}, which shows a slightly perturbed constant steady state evolving into a stripe pattern. The initial growth of the pattern takes place in the linear regime. The persistence of the pattern as it grows larger is due to the nonlinear effects (reaction). Fig.~\ref{fig:patterns_negative} illustrates an unstable steady state for $\tau<0$. Again a perturbed steady state evolves into a stripe pattern. However, the experiment only gives patterns like Fig.~\ref{fig:patterns_negative} for about $10-20\%$ of random initial conditions. The rest of the simulations give irregular cycles of blow up, which might be due to numerical instabilities when $\tau<0$.

%However, Figure~\ref{fig:patterns_negative_n} shows the pattern is stopped forming in a longer time. I refer to the pattern as `metastable pattern'.

Moreover, it is possible to get stability of the perturbed steady state when $\tau<0$ even though there is an unstable mode ($\mu<0$) in the linearized equation, as Corollary~\ref{prop:stableregion} shows. In about $80\%$ of our simulations (not shown), the initial perturbation decayed and the solution remained near the steady state until time $t=50$, as predicted qualitatively by Corollary~\ref{prop:stableregion}. In the other $20\%$ of simulations, the solution blew up chaotically, which again we think is due to numerical instabilities.

% Neither regular pattern nor metastable pattern is formed as shown in Figure~\ref{fig:stable_negative}.

%After I understand the underlying mechanism that produces natural patterns, I realize that numerical simulations enrich my understanding of the mechanism. I do some simulations for Turing instability which show Turing type pattern formation as I expected.
\begin{figure}
\centering
               \includegraphics[width=11cm, height=7cm]{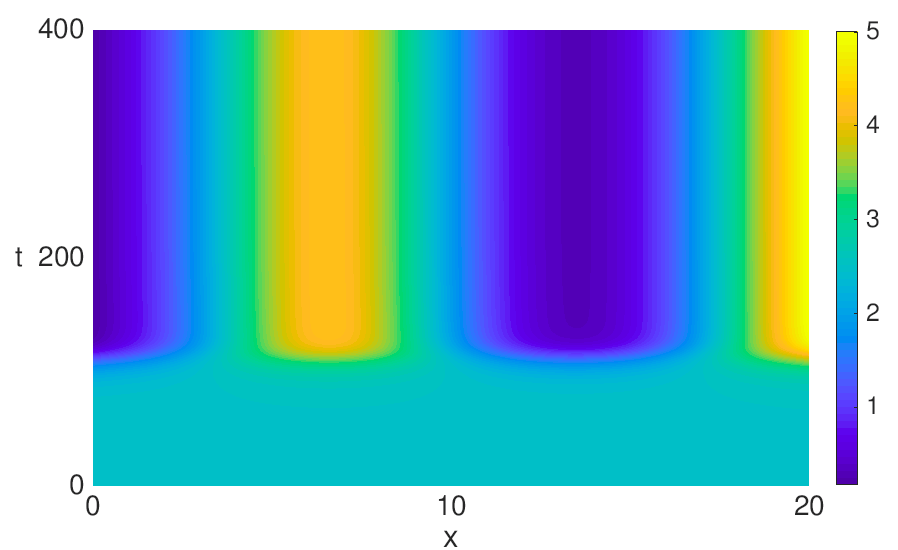}
    \caption{\label{fig:patterns_positive}Regular stripe pattern of modified fourth order Gierer--Meinhardt system \eqref{G-Mreaction} with $(k_1, k_2, k_3, k_4, k_5)=(0, 0.4, 1, 1, 1)$ when $\tau\geq 0$: the figure shows a contour plot of the activator $u$. The pattern corresponds to the point $(R, \tau)=(20, 0.5)$ in the instability region $O_{+}(1)$ (shown in Fig.~\ref{fig:three_instability_regions}) associated to the first odd eigenvalue branch $\mu_1^{\text{odd}}$ in Fig.~\ref{fig:freeeigenvalues}. (The figure uses parameter values $\vec{p}=\param$.)}
\end{figure}

\begin{figure}
\centering
        \includegraphics[width=11cm, height=7cm]{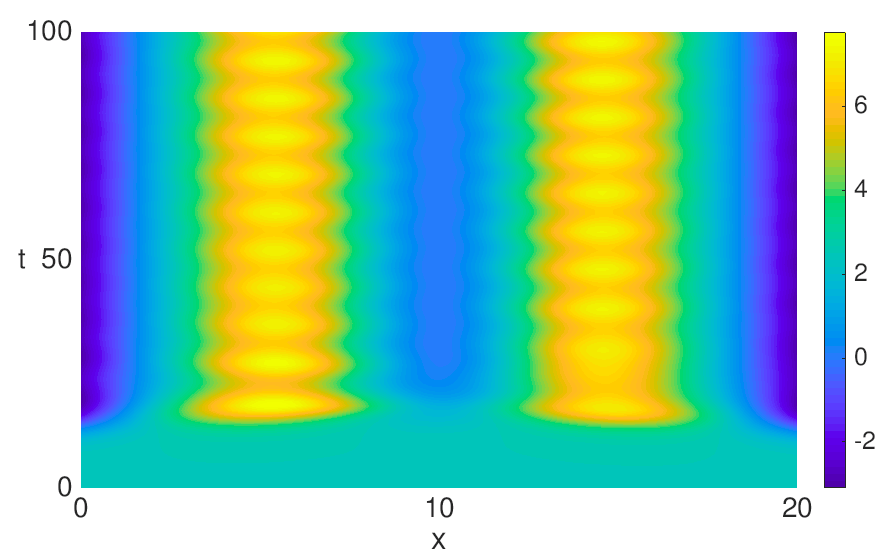}
       \caption{\label{fig:patterns_negative}Regular stripe pattern of modified fourth order Gierer--Meinhardt system \eqref{G-Mreaction} with $(k_1, k_2, k_3, k_4, k_5)=(0, 0.4, 1, 1, 1)$ when $\tau<0$: the figures shows a contour plot of the activator $u$ corresponding to the point $(R, \tau)=(20, -0.3)$ in the instability region $E_{-}(0)$ (shown in Fig.~\ref{fig:three_instability_regions}) associated to the zero-th even eigenvalue branch $\mu_0^{\text{even}}$ in Fig.~\ref{fig:freeeigenvalues}. (The figure uses parameter values $\vec{p}=\param$.) The simulation only gives a stable pattern for about $10-20\%$ of random initial conditions. The rest of the simulations give irregular cycles of blow up. Also, even in the stable pattern shown in the figure, the pattern seems to be slightly temporally periodic.}
\end{figure}

\subsection*{Plotting the instability regions}
We end the section explaining how we create the instability regions in Figs.~\ref{instability_region} and \ref{fig:three_instability_regions}. For instance, the direct formula for the bottom boundary curve of the instability region $O_{+}(l)$ is
\[
\tau=R^{-2}(\mu_l^{\text{odd}})^{-1}\left(a(\vec{p})R^4\right)
\] from Definition~\ref{def:instabilityregion}. However, it is not straightforward to obtain the curve since we do not have have an explicit formula for $\mu_l^{\text{odd}}$ as a function of $\tau$. Instead we have the parameterized curves in terms of two other parameters \cite[Theorem 5]{CC17}. Hence, it is easy to work with a parameterized formula for Figs.~\ref{instability_region} and \ref{fig:three_instability_regions}:
\begin{align*}
R(\alpha)&=\sqrt[4]{\frac{\alpha^2\beta^2}{a(\vec{p})}},\\
\tau(\alpha)&=\frac{\beta^2-\alpha^2}{\alpha\beta}\sqrt{a(\vec{p})},
\end{align*} where $\alpha$ and $\beta$ are related by $\alpha^3\tan(\alpha)=\beta^3 \tanh(\beta), l\pi\leq \alpha < (2l+1)\pi/2$ from \cite[Lemma 2]{CC17}. The point is that 
\begin{align*}
\tau(\alpha)R(\alpha)^2&=\beta^2-\alpha^2,\\
a(\vec{p})R(\alpha)^4&=\alpha^2\beta^2,
\end{align*} and so by the parameterization in \cite[Theorem 5]{CC17} we see that $a(\vec{p})R(\alpha)^4$ equals the $\mu$-value corresponding to the $\tau$-value $\tau(\alpha)R(\alpha)^2$, which means 
\[
a(\vec{p})R(\alpha)^4=\mu_l^{\text{odd}}\left((-1, 1), \tau(\alpha)R(\alpha)^2\right)
\] as we want.

The $``-"$ regions are special since each eigenvalue branch in the lower half of the spectral plane consists of infinitely many different parameterizations \cite[Theorem 11 and Lemma 13]{CC17}, whereas $``+"$ regions are given by eigenvalue branches in the upper half of the spectral plane which consist of a single parameterization \cite[Theorem 5]{CC17}. So the boundary curves of $O_{-}(l)$ and $E_{-}(l)$ are made up with infinitely many parameterizations.

%Remark. If $\eqref{firstcond}$, $\eqref{secondcond}$, and $\eqref{thirdcond}$ are satisfied, then for $some$ domain, we can get instability. But these conditions are independent of which diffusion operator we are looking at and only include information about diffusion ratio, $d$, and the reaction functions $f$ and $g$. But for $\eqref{firstcond}$, $\eqref{secondcond}$, and $\circled{5}$ to be satisfied, we require information about spectrum of diffusion operator, that is, we need $\mu_1^\ast < 0$. This part is related to next section.\\

\section{\bf Proof of Theorem~\ref{thm:turing}}\label{sec:pfthm_turing}

\begin{proof}
When $\tau\geq 0$ we get the same conditions for Turing instability as when Laplacian diffusion is used \cite[Section $2.3$]{M022}, namely conditions \eqref{firstcond}--\eqref{fourthcond}. We give this proof below, since the later parts of the proof must be modified when $\tau<0$. 

Conditions \eqref{firstcond} and \eqref{secondcond} come from requiring linear stability of the ODE system in the absence of any spatial variation, as we now explain. Without spatial variation $u$ and $v$ satisfy
\begin{equation*}
u_{t}= f(u, v), \quad v_{t}= g(u, v).
\end{equation*}
First, we linearize the system about the constant steady state $(u_0, v_0)$: set $\vec{w}= \begin{pmatrix}u-u_0\\v-v_0\end{pmatrix}$, so that for small $|\vec{w}|$, 
\begin{align*}
\vec{w}_t =  \left( \begin{array}{ccc}
	f_u & f_v \\
	g_u &  g_v \end{array} \right)\vec{w}
\end{align*} where the derivative matrix is evaluated at $u=u_0, v=v_0$. Look for solution of the form $\vec{w}\propto e^{\lambda t}$. The steady state $\vec{w}=0$ is linearly stable if $\text{Re}\,\lambda<0$ for each eigenvalue $\lambda$ of the derivative matrix. That is, where $\lambda$ satisfies the quadratic equation
\begin{align*}
\text{det}\left[\left(\begin{array}{ccc}
	f_u & f_v \\
	g_u &  g_v \end{array}\right)-\lambda I\right]=\lambda^2-(f_u+g_v)\lambda+^2(f_ug_v-f_vg_u)=0.
\end{align*} Hence linearly stability of the constant steady state for the ODE system is guaranteed if \eqref{firstcond} and \eqref{secondcond} hold:
\begin{align*}
&f_u + g_v < 0,\\
&f_u g_v - f_v g_u > 0.
\end{align*} 
We assume these conditions throughout the rest of the proof.

Conditions \eqref{thirdcond} and \eqref{fourthcond} come from requiring linear instability of the PDEs (including the diffusion term) at the constant steady state, as we now explain. Consider the full reaction-diffusion system \eqref{eqn:activator}--\eqref{eqn:inhibitor} and again linearize about $(u_0, v_0)$ to get
\begin{align}\label{sys:fullPDE}
\vec{w}_t= \left( \begin{array}{ccc}
	1 & 0 \\
	0 &  k \end{array} \right)(-\Delta^2+\tau \Delta)\vec{w}+ \left( \begin{array}{ccc}
	f_u & f_v \\
	g_u &  g_v \end{array} \right)\vec{w}.
\end{align} Define $\phi_{j}(x)$ to be the time-independent solution of the eigenvalue problem:
\begin{align*}
\Delta^2\phi_{j}-\tau\Delta\phi_{j}=\mu_{j}\phi_{j}, 
\end{align*} with the free boundary conditions \eqref{BC1}--\eqref{BC2}, where $\mu_j$ is the eigenvalue. %This eigenvalue problem was investigated recently by the author and Chasman in \cite{CC17}, and we will use some of those results in this paper.

We look for a solution $\vec{w}(x, t)$ of \eqref{sys:fullPDE} in the separated form
\begin{align*}
\vec{w}(x, t)=\sum_{j}\left( \begin{array}{ccc}
	c_j \\
	d_j \end{array} \right)e^{\lambda_j t}\phi_{j}(x),
\end{align*} where $c_j$'s and $d_j$'s are constants. Note that the growth rate $\lambda_j$ informs us about the stability of the homogeneous steady state with respect to the perturbation $\phi_j$. If the real part of  $\lambda_j$ is negative for all $j$, then any perturbations will tend to decay exponentially quickly. However, in the case that the real part of $\lambda_j$ is positive for any value of $j$, our expansion suggests that the amplitude of these modes will grow exponentially quickly and so the homogeneous steady state is linearly unstable. Substitution gives us for each $j$,
\begin{align}\label{sys:formula}
\lambda \phi_{j}=-\mu_{j}D\phi_{j}+ M \phi_{j}
\end{align} where $\displaystyle{D=\left( \begin{array}{ccc}
	1 & 0 \\
	0 &  k \end{array} \right)}$ is the diffusivity matrix and  $\displaystyle{M=\left( \begin{array}{ccc}
	f_u & f_v \\
	g_u &  g_v \end{array} \right)}$ is the stability matrix. To get a nontrivial $\phi_j$, formula \eqref{sys:formula} says $\lambda$ must be an eigenvalue of the matrix $-\mu_j D+M$, and so
\begin{equation*}
\text{det}[\lambda I + \mu_{j}D- M]=0.
\end{equation*} Hence we get the eigenvalues $\lambda(\mu_{j})$ as functions of the wavenumber $\mu_{j}$, as the two roots of
\begin{align}
& \lambda^2 + F(\mu_{j})\lambda + H(\mu_{j})=0,\label{chareq}\\ 
 F(\mu_{j})&\overset{\text{def}}=\mu_{j}(1+k)-(f_u+g_v)<0,\notag\\%\label{eqnf}\\
H(\mu_{j})&\overset{\text{def}}= k\mu_{j}^2-(kf_u+g_v)\mu_{j}+(f_ug_v-f_vg_u). \notag %\label{eqnh}
\end{align} For the steady state to be unstable to spatial perturbation, we require
\begin{align*}
\text{Re}\,[\lambda(\mu_{j})]>0 \qquad \text{for some $j\neq 0$}.
\end{align*} Recall that a quadratic equation with real coefficients has a root with positive real part if and only if either the sum of the roots is positive or the product of the roots is negative. Applied to the quadratic \eqref{chareq}, that means we want $F(\mu_j)<0$ or $H(\mu_{j})<0$, for some $j\geq1$. See Fig.~\ref{fig:unstable_band}.

\renewcommand{\thefigure}{\arabic{figure}}
\renewcommand{\thesubfigure}{(\roman{subfigure})}

\begin{figure}%
\centering
\subfigure[][]{%
\label{fig:unstable_band_1}%
        \includegraphics[width=7.5cm, height=5cm]{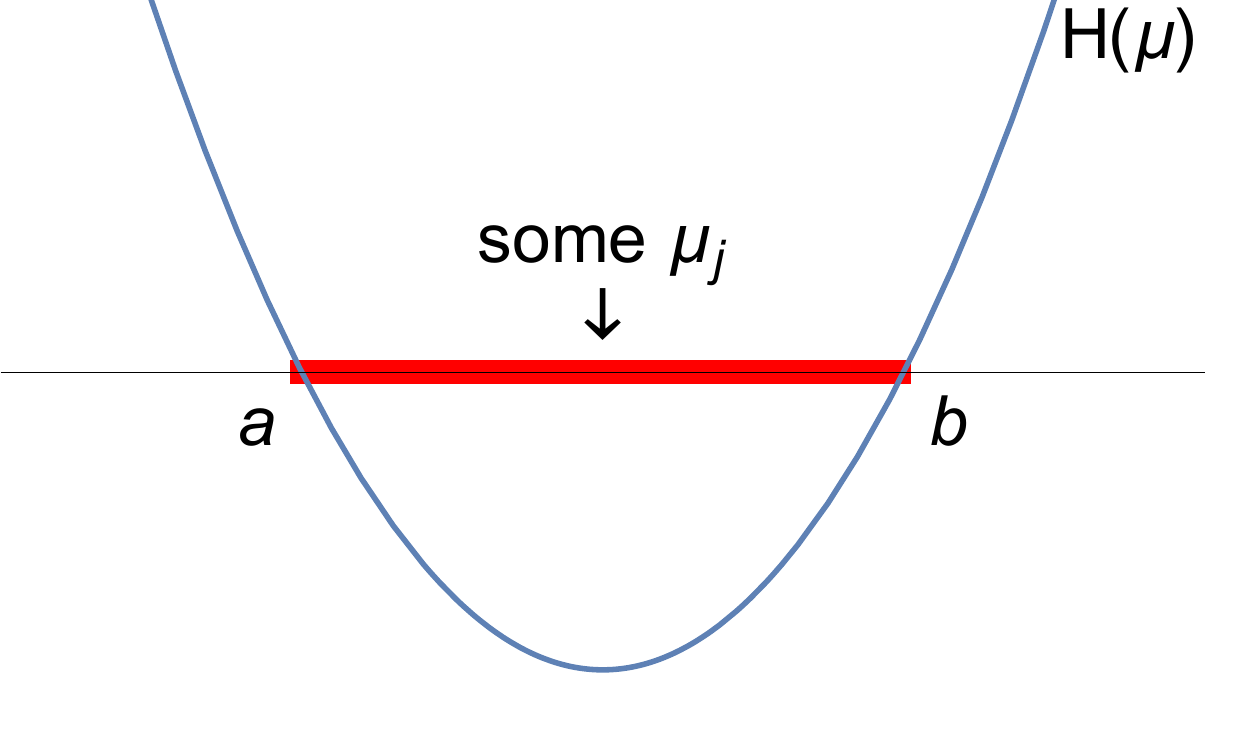}}
         \subfigure[][]{%
        \label{fig:unstable_band_2}%
\includegraphics[width=7.5cm, height=5.7cm]{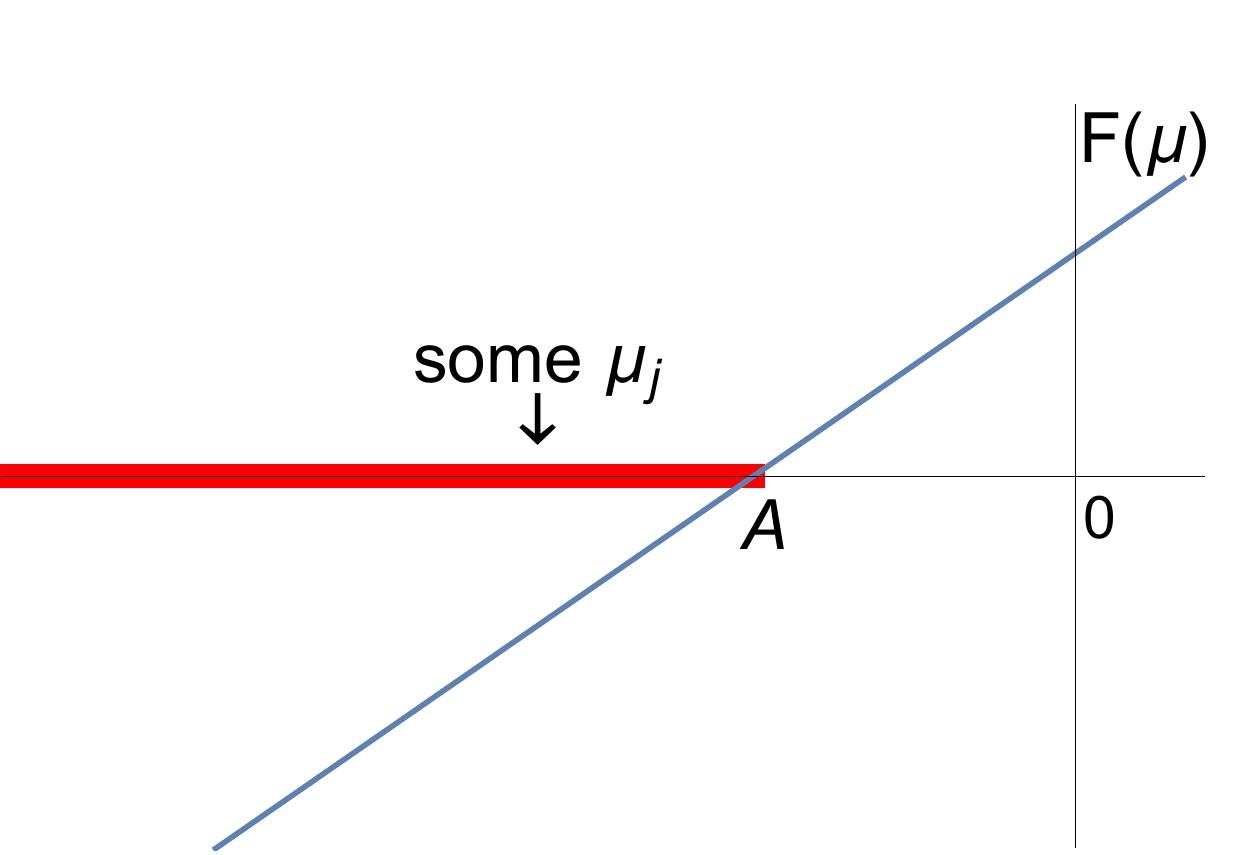}}%
\caption{There are two different ways that we can get instability. For the steady state to be unstable to spatial perturbation, we require $\text{Re}\,[\lambda(\mu_{j})]>0$ for some $j\neq 0$. This can happen if either \subref{fig:unstable_band_1} $H(\mu_{j})\overset{\text{def}}=k\mu_{j}^2-(kf_u+g_v)\mu_{j}+(f_ug_v-f_vg_u)<0$ or \subref{fig:unstable_band_2} $F(\mu_{j})\overset{\text{def}}=\mu_{j}(1+k)-(f_u+g_v)<0$ for some $j\neq 0$, where $a, b, A$ in the graphs are the intercepts along the $\mu$-axis. The figure shows general possible ways where instability can occur. Obviously, case \subref{fig:unstable_band_2} cannot happen if all the eigenvalues $\mu_j$ are positive.}%
\label{fig:unstable_band}%
\end{figure}

Now, we will consider the cases $\tau\geq0$ and $\tau<0$ separately. When $\tau<0$, the $2$nd order ``backwards" diffusion $\displaystyle{\frac{\partial u}{\partial t} = \tau \Delta u}$ is ill-posed, meaning the $2$nd order term destabilizes the system whereas the $4$th order term $\Delta^2u$ stabilizes the system. In other words, two different types of diffusion compete when the tension parameter $\tau$ is negative. However, such competition does not happen in $\tau\geq 0$ case. The $4$th and $2$nd order diffusion operators are each well-posed when $\tau\geq 0$. The case $\tau\geq 0$ is very similar to the traditional Turing analysis with the Laplacian diffusion, since all eigenvalues $\mu_j$ of \eqref{DE} are positive by the Rayleigh Quotient in Definition~\ref{def:muast}. On the other hand, the case $\tau<0$ is different from the traditional Turing instability since there are some negative eigenvalues.

\renewcommand{\thefigure}{\arabic{figure}}
\renewcommand{\thesubfigure}{(\roman{subfigure})}

\begin{figure}%
\centering
\subfigure[][]{%
\label{fig:unstable_band_negative}%
       \includegraphics[width=7.5cm, height=5cm]{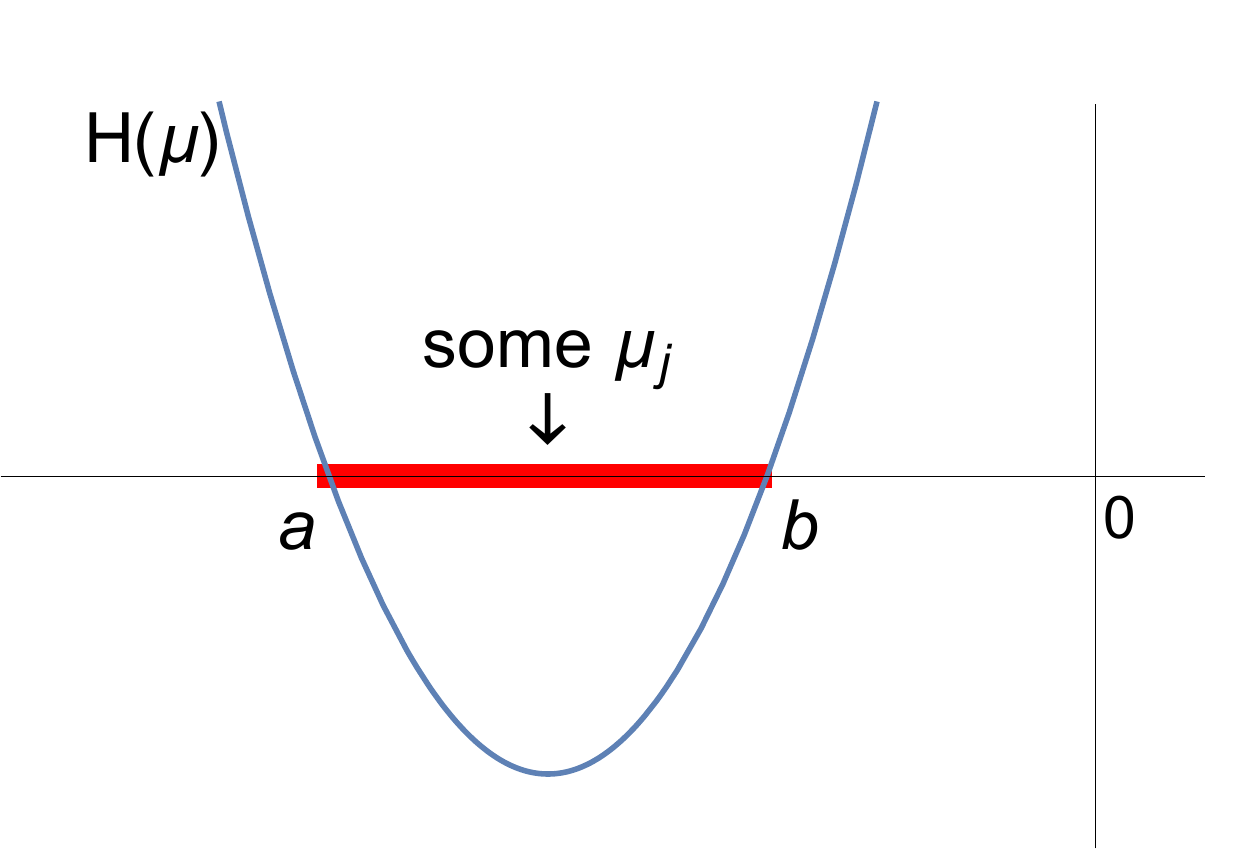}}
         \subfigure[][]{%
        \label{fig:unstable_band_positive}%
\includegraphics[width=7.5cm, height=5cm]{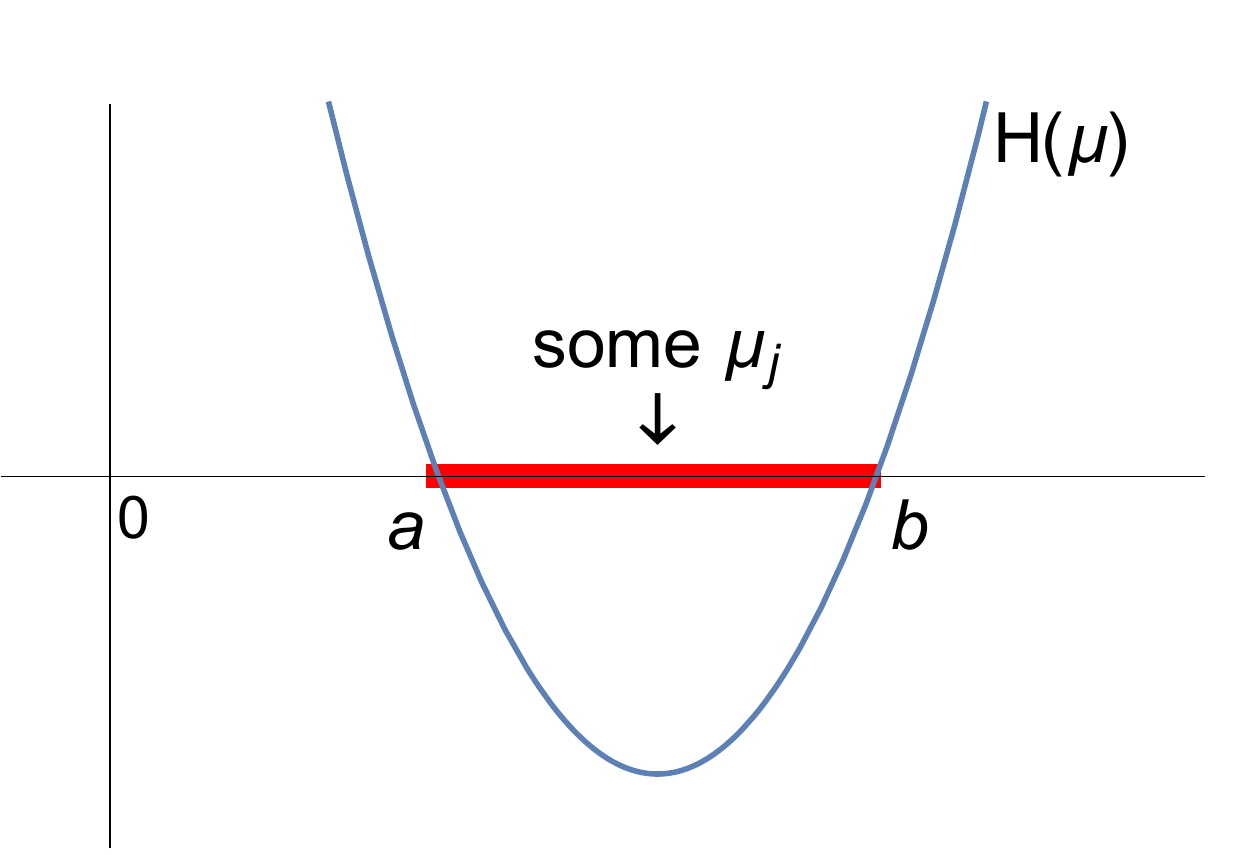}}%
\caption{The two ways to get $H(\mu_j)<0$. The distinct roots $a$ and $b$ of $H(\mu)=0$ must have the same sign since the vertical intercept is $f_ug_v-f_vg_u>0$ from the condition \eqref{secondcond}. When $\tau\geq0$, figure \subref{fig:unstable_band_positive} is the only possibility, because all $\mu_j$ are nonnegative. When $\tau<0$, there exists negative eigenvalues of $\Delta^2u-\tau\Delta u=\mu u$ and so both \subref{fig:unstable_band_negative} and \subref{fig:unstable_band_positive} are possible.}%
\label{fig:unstable_band_a}%
\end{figure}

\emph{Case $\tau\geq 0$.} When $\tau\geq 0$, the fact that $\mu_{j}$ is always positive (from Rayleigh Quotient in Definition~\ref{def:muast}) and $f_u+g_v<0$ from the stability condition \eqref{firstcond} mean $F(\mu_j)\geq 0$. So $\text{Re}\,[\lambda(\mu_{j})]>0$ if and only if $H(\mu_{j})<0$ for some $j$. See Fig.~\ref{fig:unstable_band_positive}. Since $f_u g_v - f_v g_u > 0$ by \eqref{secondcond}, we see $H(\mu_{j})<0$ if and only if \eqref{thirdcond}--\eqref{fourthcond} hold, meaning $H(\mu_j)=0$ has distinct positive roots:
\begin{align*}
(kf_u+g_v)^2-4k(f_ug_v-f_vg_u)&>0,\\
kf_u+g_v&>0,
\end{align*} since $\mu_{j}\geq0$ for all $j$. The first condition is a discriminant requirement. Hence, to be in Turing space, conditions $\eqref{thirdcond}$--$\eqref{fourthcond}$ are necessary and sufficient, when \eqref{firstcond}--\eqref{secondcond} hold.

\emph{Case $\tau< 0$.} Now, we will consider $\tau<0$. The difference from $\tau\geq 0$ comes from the fact that $\mu_{j}$ can be positive or negative and so both cases in Fig.~\ref{fig:unstable_band} can happen in order to get $\text{Re}\,[\lambda(\mu_{j})]>0$. Therefore, to be in Turing space, either $F(\mu_j)<0$ (which is $\mu_j<A(\vec{p})$ hence $\mu_1<A(\vec{p})$) or $H(\mu_{j})<0$ are necessary to hold. 

For $H(\mu_j)<0$, $H(\mu_j)=0$ must have distinct roots that have the same sign since the vertical intercept of the quadratic $H(\mu_j)$ is $f_u g_v - f_v g_u > 0$ by \eqref{secondcond} and this gives a further necessary condition \eqref{thirdcond}:
\begin{align*}
(kf_u+g_v)^2-4k(f_ug_v-f_vg_u)>0.
\end{align*} Also, the spectrum must intersect the interval $\left(a(\vec{p}), b(\vec{p})\right)$:
\[
\mu_{j}\in\left(a(\vec{p}), b(\vec{p})\right),
\] where $a(\vec{p})$ and $b(\vec{p})$ are the distinct roots of $H(\mu_j)=0$, that is, the quantities defined in \eqref{mupm}:
\begin{equation*}
a, b=a(\vec{p}), b(\vec{p})=\frac{(kf_u+g_v)\pm\sqrt{(kf_u+g_v)^2-4k(f_ug_v-f_vg_u)}}{2k}.
\end{equation*}
Note that we have two possibilities depending on the sign of $\mu_{j}$ (see Fig.~\ref{fig:unstable_band_a}). We do not have to satisfy $\eqref{fourthcond}$ because we are allowed to have negative wavenumber $\mu_{j}$ of $\Delta^2 - \tau\Delta$. We have shown that when $\tau<0$, in addition to conditions \eqref{firstcond}--\eqref{secondcond}, either condition \eqref{thirdcond} and $ \operatorname{Spec}(\Omega, \tau) \cap\left(a(\vec{p}), b(\vec{p})\right)\neq \emptyset$ are necessary to hold, or else $\mu_1<A(\vec{p})$ is necessary, for belonging to the Turing space.

Until now, we have showed necessary conditions to be in the Turing space, when $\tau<0$. To finish proving the theorem, we have to show the conditions are sufficient. We show that $\vec{p}$ belongs to the Turing space if \begin{align*}
&\text{  either \eqref{thirdcond} holds and} \operatorname{Spec}(\Omega, \tau) \cap\left(a(\vec{p}), b(\vec{p})\right)\neq \emptyset,\\
&\text{  or $\mu_1<A(\vec{p})$},
\end{align*} when $\tau<0$ and \eqref{firstcond}--\eqref{secondcond} are assumed.

%\emph{Case $\tau\geq 0$.} If $\tau\geq0$ and $\eqref{firstcond}$--$\eqref{fourthcond}$ hold, then we get Turing instability if and only if there exists at least one eigenvalue $\mu_{j}$ of the operator $\Delta^2 - \tau\Delta$ such that $\mu_{j}$ is contained in the `unstable band':
%\[
%\mu_{j}\in\left(a(\vec{p}), b\right(\vec{p})),
%\] where $a(\vec{p})$ and $b(\vec{p})$ are the roots of $H(\mu_j)=0$, that is, the quantities defined in \eqref{mupm}:
%\begin{equation*}
%a, b=a(\vec{p}), b(\vec{p})=\frac{(kf_u+g_v)\pm\sqrt{(kf_u+g_v)^2-4k(f_ug_v-f_vg_u)}}{2k}.
%\end{equation*} This proves the theorem when $\tau\geq 0$.

Assume first we are in situation of Fig.~\ref{fig:unstable_band_1}, meaning \eqref{firstcond}--\eqref{thirdcond} hold and there exists at least one eigenvalue $\mu_{j}$ such that $\mu_{j} \in \left(a(\vec{p}), b(\vec{p})\right)$. Since $H(\mu_j)<0$, Turing instability occurs.

Assume now that we are in situation of Fig.~\ref{fig:unstable_band_2}, meaning \eqref{firstcond}--\eqref{secondcond} hold and 
\begin{align*}
\mu_1<\frac{f_u+g_v}{1+k}=A(\vec{p}).
\end{align*} Since $F(\mu_1)<0$ and so $\text{Re}\,[\lambda(\mu_1)]>0$, Turing instability occurs. These prove the theorem when $\tau<0$.

\end{proof}

\section{\bf Proof of Theorem~\ref{thm:muast} and Corollary~\ref{cor:turing}}\label{sec:pfthm_muast,pfcor_turing}
\begin{proof}[Proof of Theorem~\ref{thm:muast}]
\emph{Step 1:} Fix $\tau<0$. We will prove the theorem first on a one-dimensional interval.
Let $\Omega=(-R, R)$; an interval of length $2R$ centered at the origin. We start by finding a rescaling relation. Let $\tilde{x}=x/R$ and $v(\tilde{x})= u(x)$. Then $v$ is defined on the interval $(-1, 1)$. From the transformation,  the differential equation \eqref{DE} and the one-dimensional natural boundary conditions of the type \eqref{BC1}--\eqref{BC2} are converted into 
\begin{equation*}
v''''-\tau R^2v''=R^4\mu\left(\Omega, \tau\right)v,
\end{equation*} and
\begin{align*}
\begin{cases}
v''=0 &\quad \text{at $\tilde{x}=\pm 1$},\\
v'''-\tau R^2v'=0 &\quad \text{at $\tilde{x}=\pm 1$}.
\end{cases}
\end{align*} Changing variable like this leads to the rescaling relation:
\begin{equation}
\mu_j(\Omega, \tau) = R^{-4} \mu_j((-1, 1), \tau R^2) \label{rescalingrelation}.
\end{equation} (We rescaled since we know from \cite{CC17} how the eigenvalue $\mu_j((-1, 1), \tau R^2)$ behaves with respect to the parameter $\tau R^2$ when the domain $(-1, 1)$ is fixed.) 

Notice the following equivalent conditions, when $c>0$ is fixed:
\begin{align}
\mu_j((-R, R), \tau) &= -c\tau^2 \notag \\
\Longleftrightarrow R^{-4} \mu_j((-1, 1), \tau R^2) &= -c\tau^2 \qquad \text{from the rescaling} \notag\\
\Longleftrightarrow \mu_j((-1, 1), \tau R^2) &= -c(\tau R^2)^2 \notag \\
\Longleftrightarrow  \mu_j((-1, 1), \tilde{\tau}) &= -c \tilde{\tau}^2, \label{eqcond}
\end{align} where $\tilde{\tau}=\tau R^2$. There is at least one value $\tilde{\tau}<0$ and one index $j$ such that \eqref{eqcond} holds, since we know from \cite[Proposition $16$]{CC17} there is at least one intersection between the eigenvalue curves $\mu_j((-1, 1), \tilde{\tau})$ for $\tilde{\tau}<0$ and the parabola $y=-c\tilde{\tau}^2$. From the equivalent conditions, there is at least one value $R$ and one index $j$ such that 
\[
\mu_j((-R, R), \tau) = -c\tau^2.
\] Hence $\mu_1((-R, R), \tau) \leq -c\tau^2$. We have shown that
\begin{align*}%\label{soln}
\text{for arbitrary $c>0$, there exists $R$ such that $\mu_1((-R, R), \tau) \leq -c\tau^2$}.
\end{align*} So 
\[
\mu_1^{\ast}(\tau)=\inf_{\Omega}\mu_1(\Omega, \tau) \leq -c\tau^2.
\] Letting $c\to\infty$ shows $\mu_1^{\ast}(\tau)=-\infty$.

\emph{Step 2:} Extend to $n$-dimensional cube in $\RR^n$. 
Firstly we show extension to a $2$-dimensional square domain. Let $\Omega_1=(-R, R)$ and $\Omega_2=(-R, R) \times (-R, R)$. The first eigenvalue $\mu_1$ is the minimum of Rayleigh quotient over the space of all functions $u\in H^2$, by using the Rayleigh-Ritz variational formula, that is,
\begin{align}
\mu_1(\Omega_1) &= \min_{u\in H^2(\Omega_1)} Q[u] = \min_{u\in H^2(\Omega_1)} \frac{\int_{-R}^R \left(|u''|^2 + \tau |u'|^2\right)\, dx}{\int_{-R}^R u^2\, dx},\notag\\
\mu_1(\Omega_2) &= \min_{v\in H^2(\Omega_2)} Q[v] = \min_{v\in H^2(\Omega_2)} \frac{\int_{-R}^R \int_{-R}^R \left(|D^2 v|^2 + \tau |\nabla v|^2\right)\, dx dy}{\int_{-R}^R\int_{-R}^R v^2\, dx dy}.\label{RQ2}
\end{align} We can take a function $u(x)$ of one variable in $H^2(\Omega_1)$ and regard it as a function of two variables, for instance, $v(x, y)=u(x) \in H^2(\Omega_2)$. Hence we obtain $H^2(\Omega_1)\subset H^2(\Omega_2)$. By taking minimum of each Rayleigh quotient $Q$ we get
\begin{align*}
\mu_1(\Omega_2) &= \min_{v\in H^2(\Omega_2)} Q[v] \leq \min_{u\in H^2(\Omega_1)} Q[u]=\mu_1(\Omega_1),
\end{align*} where the second integral of the right side of \eqref{RQ2} is cancelled because nothing depends on $y$ and so it comes down to the case of the first eigenvalue in one-dimensional domain. After taking infimum of $\mu_1$ and together with the above observation of one-dimensional case, we conclude that
\begin{align*}
\mu_1^{\ast}(\tau) = \inf_{\Omega_2} \mu_1(\Omega_2) = -\infty.
\end{align*} It is straightforward to generalize $2$-dimensional square case to $n$-dimensional cubes.

\end{proof}
%\begin{remark}
%We observed the advantage of taking the natural (free) boundary conditions when we modify a function of one variable in $H^2$. It is easily regarded as a function of two variables of $H^2$ since we do not have to consider boundary conditions in other direction.
%\end{remark}

\begin{proof}[Proof of Corollary~\ref{cor:turing}]
The point of Theorem~\ref{thm:muast} is that if $\tau<0$ then there exist domains that have arbitrarily negative value of $\mu_1$. Hence the condition
 \begin{align*}%\label{fifthcond}
\mu_1(\Omega, \tau)< A(\vec{p})
\end{align*} in Theorem~\ref{thm:turing} holds for some domain $\Omega$. Together with the hypotheses that $\vec{p}$ satisfies conditions \eqref{firstcond}--\eqref{secondcond}, we conclude by Theorem~\ref{thm:turing} that Turing instability occurs for the domain $\Omega$.
\end{proof}

\section{\bf Proof of Theorem~\ref{thm:instabilityregion} and Proposition~\ref{prop:moverightdown}}
Before we start the proof, we explain why we use the $``+"$ and $``-"$ notation for the sets $E_{\pm}$ and $O_{\pm}$. For $E_{+}(l)$ and $O_{+}(l)$ in Definition~\ref{def:instabilityregion}, the eigenvalues $\mu_l^{\text{even}}$ and $\mu_l^{\text{odd}}$, which lie between positive constants by the condition \eqref{fourthcond}, are positive. Hence, the sets $E_{+}$ and $O_{+}$ relate to eigenvalues that are in the upper half of the spectral plane. For $E_{-}(l)$ and $O_{-}(l)$ in Definition~\ref{def:instabilityregion}, the eigenvalues $\mu_l^{\text{even}}$ and $\mu_l^{\text{odd}}$ are negative because they are less than the negative constant $A(\vec{p})$, by assumption \eqref{firstcond}. Hence, the sets $E_{-}$ and $O_{-}$ relate to eigenvalues that are in the lower half of the spectral plane.

\begin{proof}[Proof of Theorem~\ref{thm:instabilityregion}]
\begin{enumerate}
\item ``$\supset$": Pick one case $E_{+}(1)$ as an example, since the other cases are similar. If conditions \eqref{firstcond}--\eqref{fourthcond} on the reaction-diffusion vector $\vec{p}$ are assumed, then we can apply Theorem~\ref{thm:turing} to show $E_{+}(1)\subset TS(\vec{p})$, as follows. Suppose $(R, \tau)\in E_{+}(1)$, so that by Definition~\ref{def:instabilityregion}, the eigenvalue branch $\mu_1^{\text{even}}$ satisfies
\begin{equation*}
R^{-4} \mu_1^{\text{even}}((-1, 1), \tau R^2) \in \left(a(\vec{p}), b(\vec{p})\right).
\end{equation*} Recall the domain $\Omega(R)$ is the interval $(-R, R)$. Together with the rescaling relation:
\begin{equation*}
\mu((-R, R), \tau) = R^{-4} \mu((-1, 1), \tau R^2)
\end{equation*} from \eqref{rescalingrelation}, we have
\[
\mu_1^{\text{even}}(\Omega(R), \tau)\in \left(a(\vec{p}), b(\vec{p})\right).
\] Hence by Theorem~\ref{thm:turing}, the reaction-diffusion vector $\vec{p}$ belongs to the Turing space $TS(\Omega(R), \tau)$, and so $(R, \tau)\in TS(\vec{p})$. We have shown
\begin{equation*}
TS(\vec{p})	\supset \bigcup_{l\geq 0} \big(E(l) \cup O(l)\big).
\end{equation*}

``$\subset$":
We will prove 
\begin{equation*}
TS(\vec{p})	\subset  \bigcup_{l\geq 0} \big(E(l) \cup O(l)\big)
\end{equation*} in the following. Suppose $(R, \tau)\in TS(\vec{p})$, where $\vec{p}\in TS\left(\Omega(R), \tau\right), R>0$. If $\tau\geq0$, from Theorem~\ref{thm:turing}, there exist some eigenvalue $\mu_j(\Omega(R), \tau)$ such that 
\begin{equation*}
\mu_j(\Omega(R), \tau)\in\left(a(\vec{p}), b(\vec{p})\right).
\end{equation*} From the analysis of the spectrum in \cite{CC17} we know that eigenvalues correspond to some $l$th branch of the spectrum $\mu_l$ (see Fig.~\ref{fig:freeeigenvalues}). Note that $a(\vec{p})>0$ from the condition \eqref{fourthcond} and so such $\mu_j(\Omega(R), \tau)$ are positive. Equivalently, there exist some $l$th even or odd eigenvalue branches $\mu_l^{\text{even}}$ or $\mu_l^{\text{odd}}$ such that   
\begin{equation*}
\mu_l^{\text{even}}(\Omega(R), \tau)\  \text{or}\  \mu_l^{\text{odd}}(\Omega(R), \tau)\in\left(a(\vec{p}), b(\vec{p})\right).
\end{equation*} Together with the rescaling relation, it is equivalent to 
\begin{equation*}
\mu_l^{\text{even}}\left((-1, 1), \tau R^2\right)\  \text{or}\  \mu_l^{\text{odd}}\left((-1, 1), \tau R^2\right)\in\left(a(\vec{p})R^4, b(\vec{p})R^4\right).
\end{equation*} Therefore, $(R ,\tau)$ belongs to some instability regions $E_{+}(l)$ or $O_{+}(l)$. The fact that $\mu_j(\Omega(R), \tau)$ is positive tells us $(R ,\tau)$ is in the $``+"$ regions. We have shown that if $(R, \tau)\in TS(\vec{p})$ with $\tau\geq 0$ then
\begin{equation*}
(R, \tau) \in \bigcup_{l\geq 0} \big(E_{+}(l) \cup O_{+}(l)\big).
\end{equation*}

Now if $\tau<0$, from Theorem~\ref{thm:turing}, either there exists some eigenvalue $\mu_j(\Omega(R), \tau)$ such that 
\begin{equation}\label{case1}
\mu_j(\Omega(R), \tau)\in\left(a(\vec{p}), b(\vec{p})\right), 
\end{equation} or else
\begin{equation}\label{case2}
\mu_1(\Omega(R), \tau)< A(\vec{p}).
\end{equation} The first case \eqref{case1} is the same as we showed when $\tau\geq0$. For the second case, recall from \cite[Section $5$]{CC17} that the first eigenvalue $\mu_1$ corresponds to the zero-th even or odd eigenvalue branch $\mu_0^{\text{even}}$ or $\mu_0^{\text{odd}}$ (shown in Fig.~\ref{fig:freeeigenvalues}). Note that $A(\vec{p})<0$ from the condition \eqref{firstcond} and so $\mu_1(\Omega(R), \tau)$ is negative. The second case \eqref{case2} is equivalent to 
\begin{equation*}
\mu_0^{\text{even}}(\Omega(R), \tau)\  \text{or}\  \mu_0^{\text{odd}}(\Omega(R), \tau)< A(\vec{p}).
\end{equation*} From the rescaling relation, 
\begin{equation*}
\mu_0^{\text{even}}\left((-1, 1), \tau R^2\right)\  \text{or}\  \mu_0^{\text{odd}}\left((-1, 1), \tau R^2\right)< A(\vec{p})R^4.
\end{equation*} Hence, $(R ,\tau)$ belongs to the instability regions $E_{-}(0)$ or $O_{-}(0)$. By combining when $\tau\geq0$ and the first and second cases of $\tau<0$, we have shown 
\begin{equation*}
(R, \tau) \in \bigcup_{l\geq 0} \big(E(l) \cup O(l)\big),
\end{equation*} where recall $E=E_{+}\cup E_{-}$, $O=O_{+}\cup O_{-}$.

\item The proof is similar to part (\eqref{part1}). ``$\supset$", except using $E_{-}(1)$ as the typical case instead of $E_{+}(1)$.
\end{enumerate}
\end{proof}

\begin{figure}
\centering
\includegraphics[scale=0.65]{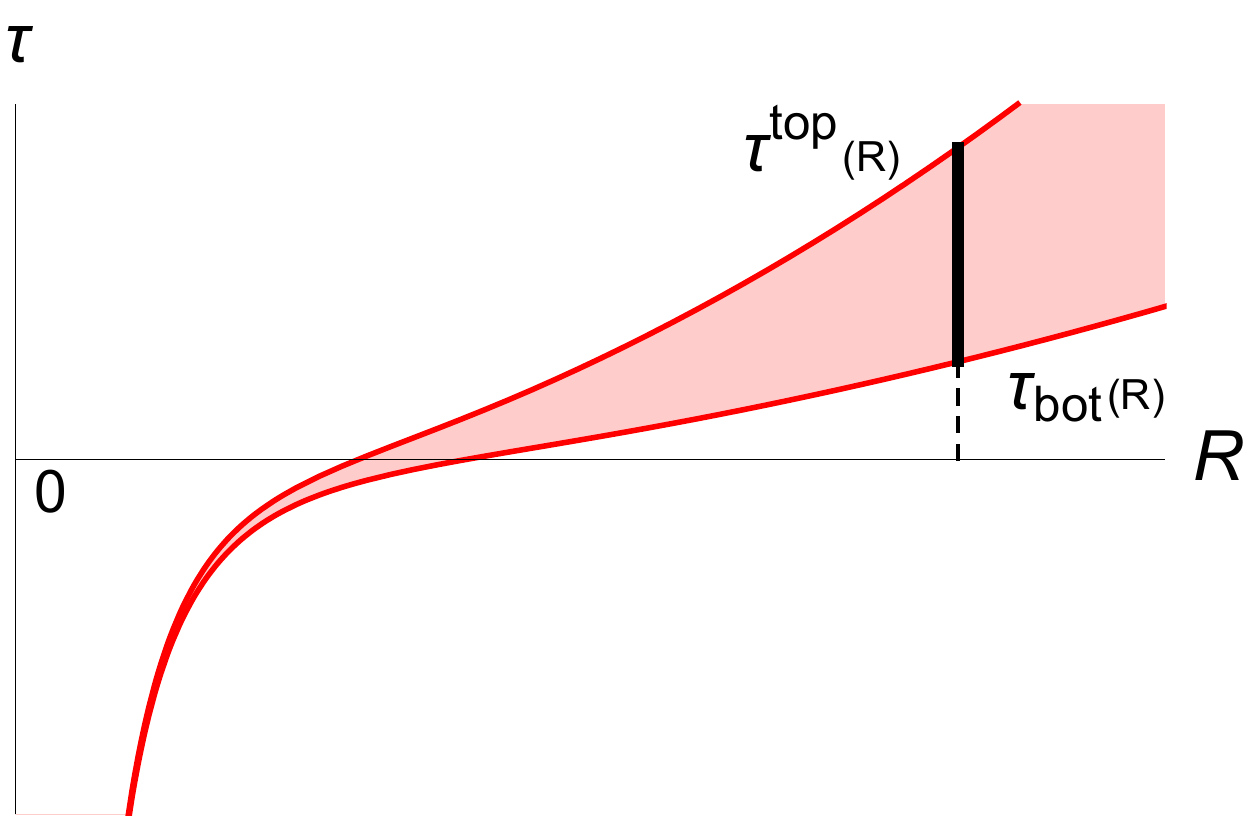}
\caption{\label{fig:move_downwards}For given $R$, there is a single interval of $\tau$ being in the instability region $E_{+}$. The top and bottom boundary curves of region $E_{+}$ are given by functions $\tau^{\text{top}}(R)$ and $\tau_{\text{bot}}(R)$.}
\end{figure}

\begin{proof}[Proof of Proposition~\ref{prop:moverightdown}]
\begin{enumerate} 
\item Fix $\vec{p}$ and $l$, and write $a=a(\vec{p}), b=b(\vec{p})$. We will prove the even case $E_{+}(l)$ and the odd case $O_{+}(l)$ is similarly obtained. Define two sets from the definition of the instability region $E_{+}(l)$ by the following:
\begin{align}
\text{Top}&= \{\left(R, \tau\right): R^{-4}\mu_{l}^{\text{even}}(\tau R^2)=b\}\label{topset},\\
\text{Bot}&= \{\left(R, \tau\right): R^{-4}\mu_{l}^{\text{even}}(\tau R^2)=a\}\label{bottomset}.
\end{align} We will prove that these sets are graphs of functions of $R$. 
First, we show that for given $R$ there is a single interval of $\tau$-values that satisfies the condition
\begin{equation*}
R^{-4}\mu_{l}^{\text{even}}(\tau R^2)\in(a, b)
\end{equation*} for being in the instability region $E_{+}(l)$ in Definition~\ref{def:instabilityregion}. The condition is equivalent to
\begin{equation*}
\tau  \in \left(R^{-2}(\mu_{l}^{\text{even}})^{-1}(aR^4), R^{-2}(\mu_{l}^{\text{even}})^{-1}(bR^4)\right),
\end{equation*} which is a single interval since $\mu_{l}^{\text{even}}(\tau)$ is a strictly increasing function \cite[Proposition $7$]{CC17} so that the inverse is uniquely defined. See Fig.~\ref{fig:move_downwards}. Hence the sets \eqref{topset} and \eqref{bottomset} are the graphs of the function:
\begin{align*}
\tau^{\text{top}}(R; l)&=R^{-2}(\mu_{l}^{\text{even}})^{-1}(bR^4),\\%\label{topcurve}\\
\tau_{\text{bot}}(R; l)&=R^{-2}(\mu_{l}^{\text{even}})^{-1}(aR^4)\notag%\label{bottomcurve}.
\end{align*} It is clear from Definition~\ref{def:instabilityregion} that these are the top and bottom boundary curves of region $E_{+}(l)$. 

Now, we will prove that $E_{+}(l)$ moves downward as $l$ increases, in the sense that
\begin{align*}
\tau^{\text{top}}(R; l)&>\tau^{\text{top}}(R; l+1)>\tau^{\text{top}}(R; l+2)>\cdots,\\
\tau_{\text{bot}}(R; l)&>\tau_{\text{bot}}(R; l+1)>\tau_{\text{bot}}(R; l+2)>\cdots.
\end{align*} Notice that for all $\tau\in\RR$,
\begin{align*}
\mu_{l}^{\text{even}}(\tau)<\mu_{l+1}^{\text{even}}(\tau)
\end{align*} as proved in \cite[Proposition $6$]{CC17}. Since the eigenvalue branches $\mu_{l}^{\text{even}}(\tau)$ are strictly increasing to infinity, we know that the inverse functions satisfy
\begin{align*}
\left(\mu_{l+1}^{\text{even}}\right)^{-1}(b)<\left(\mu_{l}^{\text{even}}\right)^{-1}(b), 
\end{align*} for fixed $b$. Therefore, we obtain
\begin{align*}
\tau^{\text{top}}(R; l+1)=R^{-2}\left(\mu_{l+1}^{\text{even}}\right)^{-1}(bR^4)<R^{-2}\left(\mu_{l}^{\text{even}}\right)^{-1}(bR^4)=\tau^{\text{top}}(R; l),
\end{align*} and similarly for the bottom curves.

%{\color{blue}(Intercepts of surface curves of instability region $E_{+}(l)$ and $O_{+}(l)$) Intercepts (where $\tau=0$) are given by the vertical intercept of $l$th eigenvalue branch of $\Delta^2-\tau\Delta$ in $(\tau, \mu)$-plane. For instance, let's consider the $l$th odd eigenvalue branch. From \cite{CC17}, the vertical intercept is
%\begin{equation}
%(0, \mu)=(0, a^4), \quad \text{where $a=l\pi+\frac{\pi}{4}+\mathcal{O}(1)$}.
%\end{equation} The asymptotic for $a$ can be improved with better control on how quickly $\tanh(a)\to 1$ as $a\to\infty$. The left intercept $(R_1, 0)$ is defined by
%\begin{equation}
%bR_1^4=\left(l\pi+\frac{\pi}{4}+\mathcal{O}(1)\right)^4.
%\end{equation} That is, $\displaystyle{R_1=\left(l\pi+\frac{\pi}{4}+\mathcal{O}(1)\right)/\sqrt[4]{b}}$. The right intercept $(R_2, 0)$ can be found similarly and two intercepts move to the right as $l$ increases. Hence we conclude that the instability region $O_{+}(l)$ moves to the right along with these intercepts as index $l$ increases.}
\item Now we consider $E_{-}(l)$ and $O_{-}(l)$, assuming conditions \eqref{firstcond}--\eqref{secondcond} hold. We will prove the even case $E_{-}(l)$ and the odd case $O_{-}(l)$ is similarly obtained. We will show the regions $E_{-}(l)$ are nested as $l$ increases in the sense that the boundary curve of $E_{-}(l)$ is nested as $l$ increases. We can express the boundary curve of $E_{-}(l)$ as the function of $R$ in a similar way to part (\ref{prop:moverightdown_1}):
\begin{equation*}
\tau^{\text{top}}(R; E_{-}(l))=R^{-2}(\mu_l^{\text{even}})^{-1}(AR^4).
\end{equation*} We will show that as $l$ increases,
\begin{align*}
\tau^{\text{top}}(R; E_{-}(l))>\tau^{\text{top}}(R; E_{-}(l+1))>\tau^{\text{top}}(R; E_{-}(l+2))>\cdots.
\end{align*} For all $\tau\in\RR$, like above we have
\begin{align*}
\mu_{l}^{\text{even}}(\tau)<\mu_{l+1}^{\text{even}}(\tau)
\end{align*} and so
\begin{align*}
\left(\mu_{l+1}^{\text{even}}\right)^{-1}(A)<\left(\mu_{l}^{\text{even}}\right)^{-1}(A), 
\end{align*} for fixed $A$. Therefore, we obtain
\begin{align*}
\tau^{\text{top}}(R; E_{-}(l+1))=R^{-2}\left(\mu_{l+1}^{\text{even}}\right)^{-1}(AR^4)<R^{-2}\left(\mu_{l}^{\text{even}}\right)^{-1}(AR^4)=\tau^{\text{top}}(R; E_{-}(l)).
\end{align*}

\end{enumerate}
\end{proof}

%\begin{figure}
%\includegraphics[scale=0.5]{Bounded_Spectrum}
%\caption{\label{boundedspectrum}The $l$th even and odd eigenvalue branches $\mu_{l}^{\text{even}}$ and $\mu_{l}^{\text{odd}}$ for $\Delta^2u-\tau\Delta u=\mu u$ are bounded below by the half parabola $P_{l}$ and bounded above by $P_{l+1}$. For instance, $\mu_{2}^{\text{even}}$ and $\mu_{2}^{\text{odd}}$ are bounded below by the half parabola $P_2$ and above by $P_3$. More of the spectrum is shown in Figure~\ref{fig:freeeigenvalues}.}
%\end{figure}

\begin{figure}
\centering
 \includegraphics[scale=0.45]{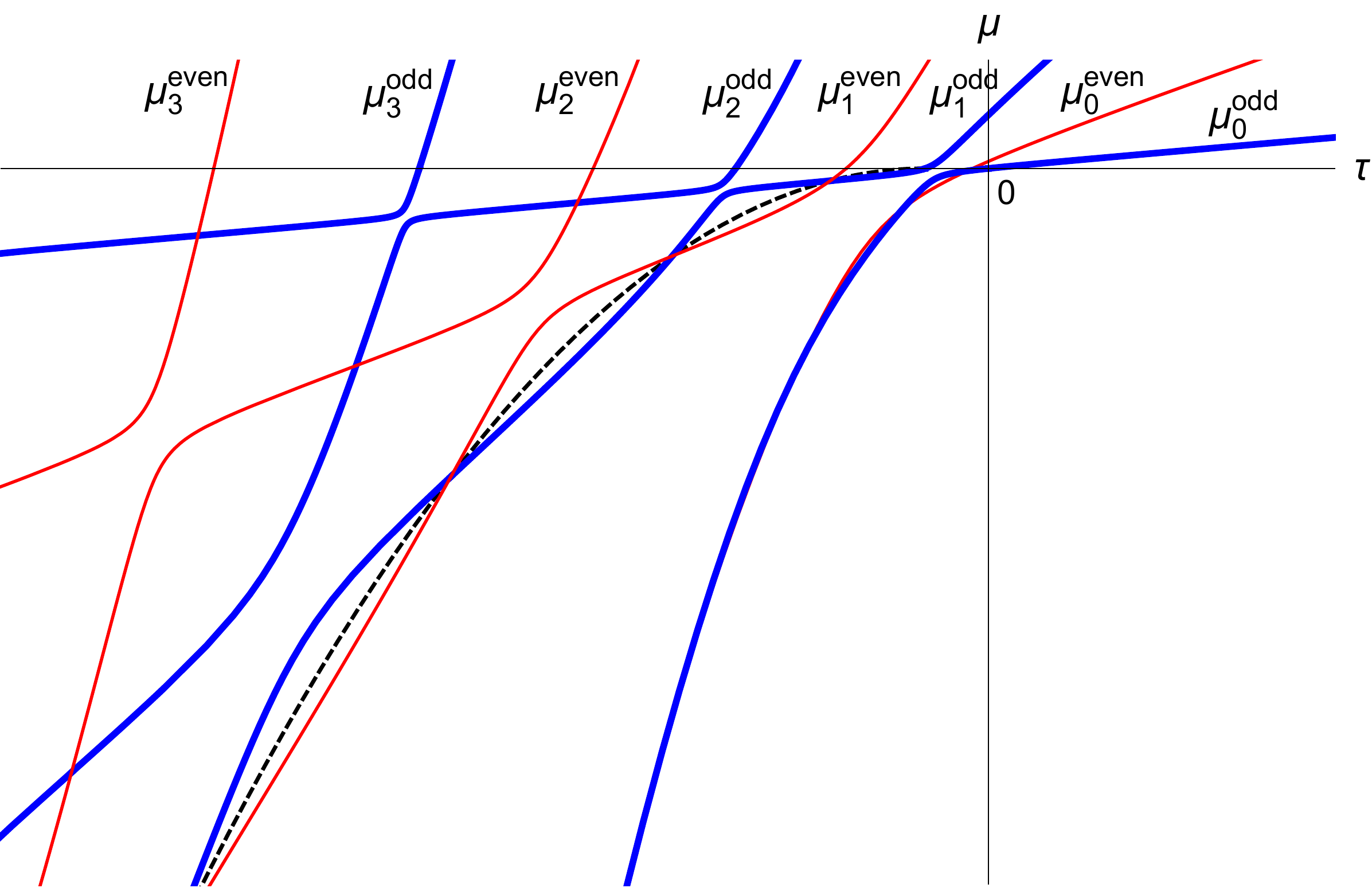}
              \caption{\label{fig:freeeigenvalues}Spectrum of $\Delta^2u-\tau\Delta u=\mu u$ with free boundary conditions. Blue curves (darker) are eigenvalue branches associated with odd eigenfunctions, red curves (lighter) are branches associated with even eigenfunctions. This spectrum was investigated recently by the author and Chasman \cite{CC17}.}
\end{figure}

\section{\bf Proof of Corollary~\ref{prop:stableregion}}

\begin{proof}[\unskip\nopunct]
To prove there exists stable region with $\tau<0$, for each fixed $\vec{p}$ satisfying \eqref{firstcond}--\eqref{fourthcond}, we study the region near the origin in Fig.~\ref{fig:three_instability_regions}. We will show:
\begin{align*}
&\text{$1$. The bottom boundary of $O_{+}(0)$ lies above the horizontal axis $\tau=0$.}\\
&\text{$2$. The boundary of $O_{-}(0)$ lies below the horizontal axis $\tau=0$.}\\
&\text{$3$. The boundary curves of $E_{+}(0)$, $E_{-}(0)$, and $O_{+}(1)$ have $\tau\to-\infty$ as $R\to 0$.}
\end{align*} Since the instability regions $E_{\pm}(l)$ and $O_{\pm}(l)$ move downwards as $l$ increases, by Proposition~\ref{prop:moverightdown}, there exist some regions near the origin that are not covered by any instability regions in $\{(R, \tau): R>0\ \text{and}\ \tau<0\}$.

\emph{Step $1$:} the bottom boundary of $O_{+}(0)$ lies above the horizontal axis $\tau=0$.

Recall the number $a=a(\vec{p})>0$ from \eqref{fourthcond}. Since $\mu_0^{\text{odd}}(\tau)$ is strictly increasing with $\mu_0^{\text{odd}}(0)=0$ \cite[Proposition 7 and Section 3]{CC17}, we have $(\mu_0^{\text{odd}})^{-1}(aR^4)>0$. Hence,
\begin{equation*}
\tau_{\text{bot}}(R; O_{+}(0))>0,
\end{equation*} which means the bottom boundary of $O_{+}(0)$ lies above the horizontal axis $\tau=0$.

\emph{Step $2$:} the boundary of $O_{-}(0)$ lies below the horizontal axis $\tau=0$. 

We know from the spectrum, the eigenvalue branch $\mu_0^{\text{odd}}$ is approximately a straight line $\mu_0^{\text{odd}}(\tau)\simeq\pi^2\tau/4$ near the origin \cite[Section $3. 3$]{CC17}. Hence the boundary curve of $O_{-}(0)$ satisfies 
\begin{align*}
\lim_{R\to0^{+}} \frac{\tau^{\text{top}}(R; {O_{-}(0)})}{R^2}= \lim_{R\to0^{+}} \frac{(\mu_{0}^{\text{odd}})^{-1}(AR^4)}{R^4}= \frac{4A}{\pi^2}<0
\end{align*} because $A<0$ by \eqref{firstcond}. Since the number $4A/\pi^2$ is negative, the limit shows that the curve $\tau^{\text{top}}(R; {O_{-}(0)})$ lies below some negative quadratic, near the origin.

\emph{Step $3$:} the boundary curves of $E_{+}(0)$, $E_{-}(0)$, and $O_{+}(1)$ have $\tau\to-\infty$ as $R\to 0$.

As $R\to0^{+}$, the limit of the upper boundary curve of $E_{+}(0)$ is
\begin{align*}
\lim_{R\to0^{+}} \tau^{\text{top}}(R; {E_{+}(0)})= \lim_{R\to0^{+}} R^{-2}(\mu_{0}^{\text{even}})^{-1}(bR^4)=-\infty,
\end{align*} since the inverse function $(\mu_{0}^{\text{even}})^{-1}(0)=-\pi^2/4$ from the spectrum of $\Delta^2u-\tau\Delta u=\mu u$ \cite[Sections $3$ and $4$]{CC17}. The same is true for $O_{+}(1)$, since $(\mu_{1}^{\text{odd}})^{-1}(0)=-\pi^2$.

As $R\to0^{+}$, the limit of the boundary curve of $E_{-}(0)$ is
\begin{align*}
\lim_{R\to0^{+}} \tau^{\text{top}}(R; {E_{-}(0)})= \lim_{R\to0^{+}} R^{-2}(\mu_{0}^{\text{even}})^{-1}(AR^4)=-\infty,
\end{align*} since the inverse function $(\mu_{0}^{\text{even}})^{-1}(0)=-\pi^2/4$.

\end{proof}

Corollary~\ref{prop:stableregion} tells us that there is some region that does not belong to any of Turing space $TS(\vec{p})$, as shown by the unshaded regions in Fig.~\ref{fig:three_instability_regions}.

\section{\bf Periodic boundary conditions in one dimension: Turing instability regions for the fourth order diffusion operator $\Delta^2-\tau\Delta$}
In this section, we will illustrate the Turing instability region of the periodic boundary conditions, which has similar shape with the region for the free boundary conditions. The eigenvalue problem is
\begin{equation*}
u''''-\tau u''=\mu u
\end{equation*} for $-R<x<R$.
We can explicitly express the spectrum of periodic case in one dimension for $-1<x<1$ as
\begin{equation*}%\label{periodicspectrum}
\mu_{l}^{\text{per}}(\tau)=(l\pi)^4+\tau(l\pi)^2, \quad l\geq0, 
\end{equation*} where eigenfunctions can be taken as the even function $u_{\text{e}}(x)=\cos(l\pi x)$ or the odd function $u_{\text{o}}(x)=\sin(l\pi x)$. Note that all the eigenvalues have multiplicity $2$, except for $l=0$. In the periodic case, we do not need to separate the even and odd instability regions since these regions are the same because eigenvalues associated to even and odd eigenfunctions are the same. We illustrate the spectrum in $(\tau, \mu)$-plane, as shown in Fig.~\ref{fig:periodiceigenvalues}. Each branch is a straight line. We see there is a parabola $\mu=-(\tau+\pi^2)^2/4$ on which the intersections of consecutive eigenvalue branches lie. The same parabola occurs also in the spectrum of the free boundary conditions as the parabola on which the intersections of the first even and odd eigenvalue branch $\mu_{1}^{\text{even}}$ and $\mu_{1}^{\text{odd}}$ lie \cite[Proposition $12$]{CC17}. On top of that, we see the spectrum of periodic boundary conditions and the spectrum of free boundary conditions behave in asymptotically similar way: compare Fig.~\ref{fig:periodiceigenvalues} and Fig.~\ref{fig:freeeigenvalues}. Actual crossings occur in the spectrum of periodic boundary conditions, whereas there are barely-avoided crossings along eigenvalue branches for free boundary conditions. A pattern of barely-avoided crossings leads to a pattern of nearly-linear segments in the free case, while the periodic spectrum contains actual line segments. Similar spectral behavior of periodic and free boundary conditions should generate similar shape of the instability regions (Fig.~\ref{fig:three_instability_regions} and Fig.~\ref{fig:periodicinstabilityregions}).
\begin{figure}
 \includegraphics[scale=0.5]{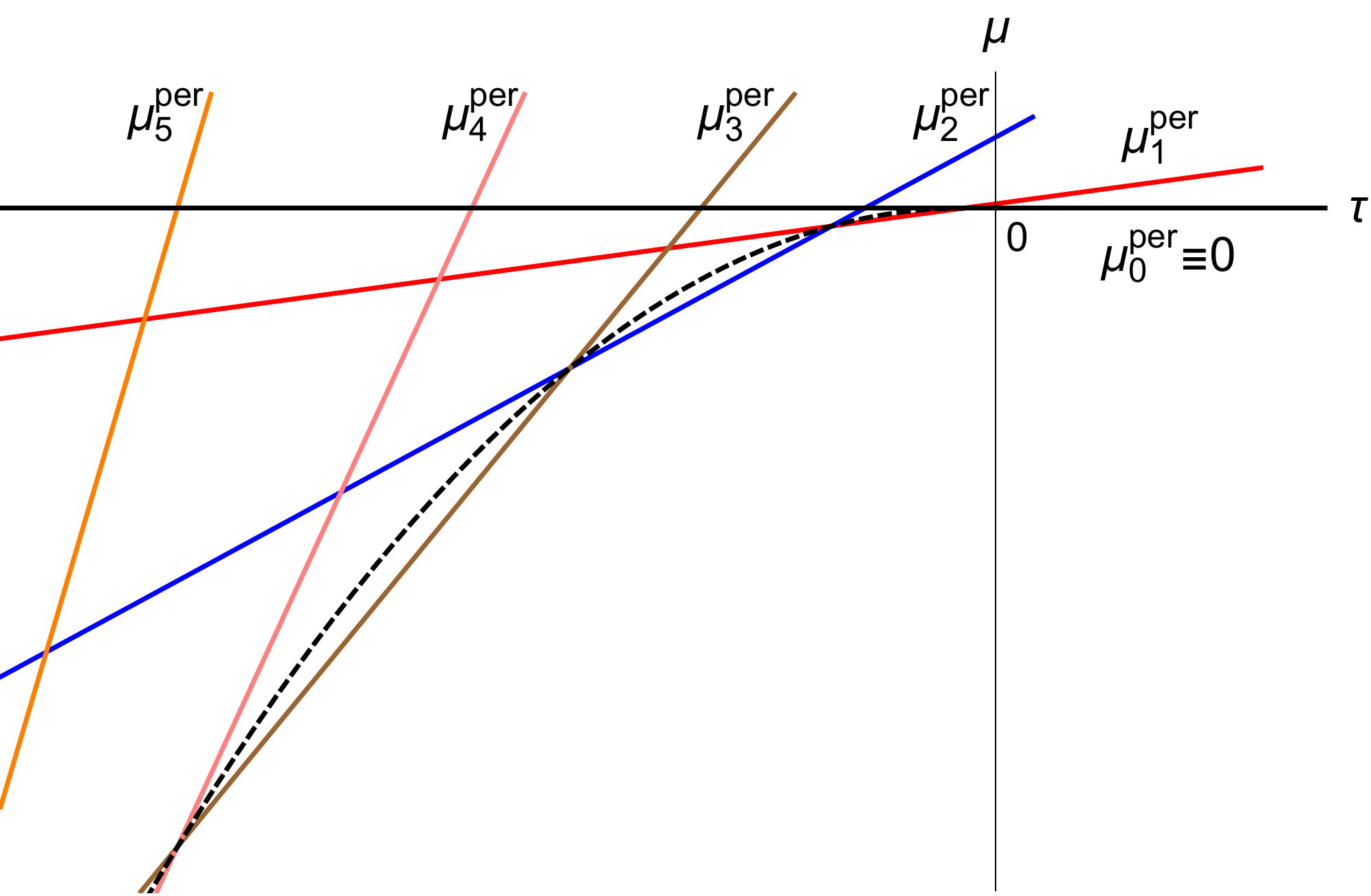}
\caption{\label{fig:periodiceigenvalues}Spectrum of $u''''-\tau u''=\mu u$ on $(-1, 1)$ with periodic boundary conditions. The dashed curve is the parabola $\mu=-(\tau+\pi^2)^2/4$ on which the intersections of consecutive eigenvalue branches lie.}
\end{figure}

Assume conditions \eqref{firstcond}--\eqref{fourthcond} hold on the reaction-diffusion vector $\vec{p}$. With a Turing analysis similar to Definition~\ref{def:instabilityregion} and Theorem~\ref{thm:instabilityregion}, we can express the instability region of periodic boundary conditions explicitly as follows, for the interval $(-R, R)$:
\begin{align*}
I_{+}^{\text{per}}(l) &= \{\left(R, \tau\right): aR^4<(l\pi)^4+\tau R^2(l\pi)^2< bR^4\},\\
I_{-}^{\text{per}}(l) &= \{\left(R, \tau\right): (l\pi)^4+\tau R^2(l\pi)^2< AR^4\ \text{and}\ \tau<0\}.
\end{align*} %We use the notation $I^{\text{per}}=I_{+}^{\text{per}} \cup I_{-}^{\text{per}}$.

The instability regions of the first four eigenvalue branches of periodic case are illustrated in Fig.~\ref{fig:periodicinstabilityregions} in the $(R, \tau)$-plane.

\begin{figure}
\centering
 \includegraphics[scale=0.45]{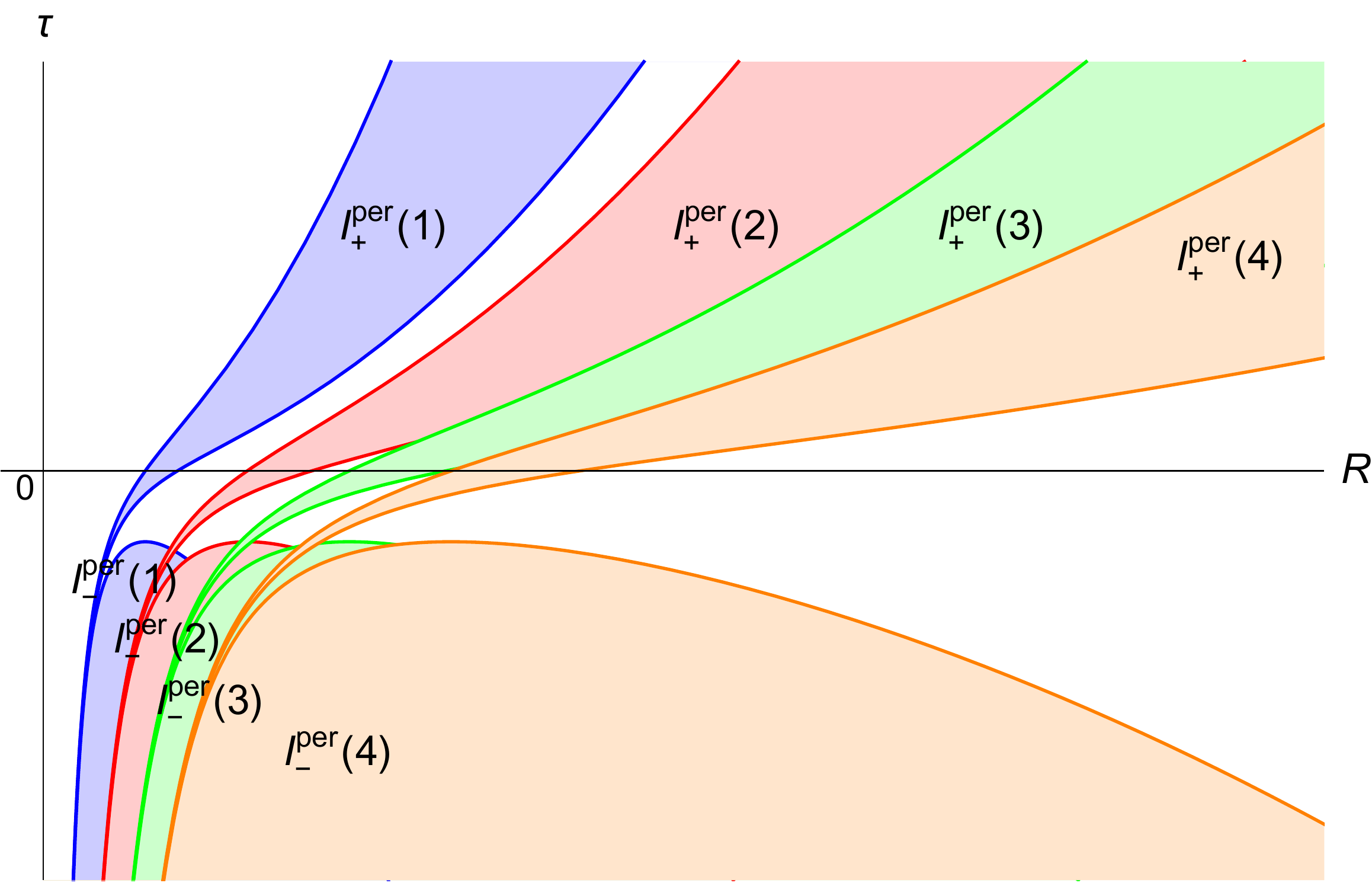}
              \caption{\label{fig:periodicinstabilityregions}The instability regions associated to periodic boundary conditions on the interval $(-R, R)$, for eigenvalue branches $l=1, \dots, 4$, assuming conditions \eqref{firstcond}--\eqref{fourthcond} for the reaction-diffusion vector $\vec{p}$. (The figure uses parameter values $\vec{p}=\param$.) Each colored region describes the instability region associated to the corresponding colored eigenvalue branch in Fig.~\ref{fig:periodiceigenvalues}. Points $(R, \tau)$ in shaded regions belong to the Turing space $TS(\vec{p})$. }
\end{figure}

\section{Acknowledgments}
This research was supported by the University of Illinois Research Board (award RB17002).


\begin{thebibliography}{99}
\bibitem{AKS11} S.\ Aly, I.\ Kim and D.\ Sheen. \emph{Turing instability for a ratio-dependent predator-prey model with diffusion.}
Appl. Math. Comput. 217 (2011) 7265--7281.

\bibitem{ABM17} M.\ S.\ Ashbaugh, R.\ D.\ Benguria and R.\ Mahadevan. \emph{Minimization of the
lowest eigenvalue of the vibrating clamped plate under compression.} Forthcoming 2018.

\bibitem{BBDK94} A.\ L.\ Bertozzi, M.\ P.\ Brenner, T.\ F.\ Dupont and L.\ P.\ Kadanoff. \emph{Singularities and similarities in interface flows.}
Trends and perspectives in applied mathematics, Appl. Math. Sci., 100, Springer, New York (1994), pp. 155--208. 

\bibitem{BP98} A.\ L.\ Bertozzi and M.\ C.\ Pugh. \emph{Long-wave instabilities and saturation in thin film equations.}
Comm. Pure Appl. Math. 51 (1998), no. 6, 625--661. 

%\bibitem{BMS99} I. Birindelli, E. Mitidieri, G. Sweers, \emph{The existence of the principal eigenvalue for cooperative elliptic systems in a general domain.} 
%Differential Equations 35 (1999), no. 3, 326--334.

\bibitem{C11} L.\ M.\ Chasman. \emph{An isoperimetric inequality for fundamental tones of free plates.}
Comm. Math. Phys. 303 (2011), no. 2, 421--449. 

\bibitem{CC17} L.\ M.\ Chasman and J.\ Chung. \emph{Spectrum of the free rod under tension and compression.} Appl. Anal., to appear (2018). \doi{10.1080/00036811.2018.1451639}

\bibitem{CGM99} E.\ J.\ Crampin, E.\ A.\ Gaffney and P.\ K.\ Maini. \emph{Reaction and diffusion on growing domains: scenarios for robust pattern formation.}
Bull. Math. Biol. 61(6) (1999), 1093--1120.

\bibitem{DMO94} R.\ Dillon, P.\ K.\ Maini and H.\ G.\ Othmer. \emph{Pattern formation in generalized Turing systems. $\textrm{I}$. Steady-state patterns in systems with mixed boundary conditions.} 
J. Math. Biol. 32 (1994), no. 4, 345--393. 

%\bibitem{DMS08} G. Derks, S. Maad, B. Sandstede, \emph{Perturbations of embedded eigenvalues for the bilaplacian on a cylinder.}
%Discrete Contin. Dyn. Syst. 21 (2008), no. 3, 801--821. 



%\bibitem{DMS11} G. Derks, S. Maad Sasane, B. Sandstede, \emph{Perturbations of embedded eigenvalues for the planar bilaplacian.}
%J. Funct. Anal. 260 (2011), no. 2, 340--398. 

%\bibitem{GS07} A. Ghazaryan, B. Sandstede, \emph{Nonlinear convective instability of Turing-unstable fronts near onset: a case study.}
%SIAM J. Appl. Dyn. Syst. 6 (2007), no. 2, 319--347.

\bibitem{GM72} A.\ Gierer and H.\ Meinhardt. \emph{A theory of biological pattern formation.}
Biol. Cybern (1972). 12, 30--39.

\bibitem{KG} V.\ Klika and E.\ A.\ Gaffney. \emph{History dependence and the continuum approximation breakdown: the impact of domain growth on Turing's instability.}
Proc. A. 473 (2017), no. 2199, 20160744, 19 pp.

\bibitem{LHL09} J.\ M.\ Lee, T.\ Hillen and M.\ A.\ Lewis. \emph{Pattern formation in prey-taxis systems.}
J. Biol. Dyn. 3 (2009), no. 6, 551--573.

\bibitem{L94} M.\ A.\ Lewis. \emph{Spatial coupling of plant and herbivore dynamics: the contribution of herbivore dispersal to transient and persistent ``waves" of damage.}
Theor. Popul. Biol. 45 (1994), 277--312.


\bibitem{LMP12}  M.\ A.\ Lewis, P.\ K.\ Maini and S.\ Petrovskii. \emph{Dispersal, Individual Movement and Spatial Ecology: A Mathematical Perspective.}
Lecture Notes in Mathematics, 2071. Springer, Heidelberg, 2013.

\bibitem{MGM10} A.\ Madzvamuse, E.\ A.\ Gaffney and P.\ K.\ Maini. \emph{Stability analysis of non-autonomous reaction-diffusion systems: the effects of growing domains.}
J. Math. Biol. 61 (2010), no. 1, 133--164.

\bibitem{MWBG12} P.\ K.\ Maini, T.\ E.\ Woolley, R.\ E.\ Baker, E.\ A.\ Gaffney and S.\ S.\ Lee. \emph{Turing's model for biological pattern formation and the robustness problem.} 
Interface Focus 2, 487 (2012).

%\bibitem{M021} J. D. Murray, \emph{Mathematical biology. $\textrm{I}$. (English summary) 
%An introduction. Third edition. Interdisciplinary Applied Mathematics, 17.}
%Springer-Verlag, New York, 2002. xxiv+551 pp.

\bibitem{M022} J.\ D.\ Murray. \emph{Mathematical biology. $\textrm{II}$. 
Spatial models and biomedical applications. Third edition. Interdisciplinary Applied Mathematics, 18.}
Springer-Verlag, New York, 2003.

\bibitem{NKL95} M.\ Neubert, M.\ Kot and M.\ A.\ Lewis. \emph{Dispersal and pattern formation in a discrete-time predator-prey model.}
Theoretical Population Biogy, 48(1) (1995): 7--43.

\bibitem{NMMWS03} H.\ F.\ Nijhout, P.\ K.\ Maini, A.\ Madzvamuse, A.\ J.\ Wathen and T.\ Sekimura. \emph{Pigmentation pattern formation in butterflies: experiments and models.}
Comptes Rendus Biologies 2003; 326(8): 717--727. PMID: 14608692

\bibitem{PSPBM04} R.\ G.\ Plaza, F.\ S$\acute{a}$nchez-Gardu$\tilde{n}$o, P.\ Padilla, R.\ A.\ Barrio and P.\ K.\ Maini. \emph{The effect of growth and curvature on pattern formation.}
J. Dynam. Differential Equations 16 (2004), no. 4, 1093--1121.

%\bibitem{Sweers92} G. Sweers, \emph{Strong positivity in $C(\bar{\Omega})$ for elliptic systems.}
%Math. Z. 209 (1992), no. 2, 251--271.

\bibitem{SEL97} J.\ A.\ Sherratt, B.\ T.\ Eagan and M.\ A.\ Lewis. \emph{Oscillations and chaos behind predator-prey invasion: mathematical artifact or ecological reality?}
Philos Trans. R. Soc. London Ser. B 352 (1997) 21--38.

\bibitem{Sweers(16)}
G. Sweers. \emph{On sign preservation for clotheslines, curtain rods, elastic membranes and thin plates.}
Jahresber. Dtsch. Math.-Ver. 118 (2016), no. 4, 275--320. 

\bibitem{T52} A.\ M.\ Turing. \emph{The chemical basis of morphogenesis.} 
Philos. Trans. Roy. Soc. London Ser. B 237 (1952), no. 641, 37--72. 








\end{thebibliography}
\end{document}